\documentclass[11pt]{amsart}
\usepackage{amsmath,amsfonts,amsthm,mathrsfs}
\usepackage{amssymb}

\usepackage[unicode,bookmarks]{hyperref}
\usepackage[dvipsnames]{xcolor}
\hypersetup{colorlinks=true,citecolor=NavyBlue,linkcolor=Brown,urlcolor=Orange}

\usepackage{tikz-cd,adjustbox}
\usepackage[arrow,curve,matrix]{xy}

\usepackage[margin=1in]{geometry}

\newif\ifdraft
\draftfalse

\counterwithin{equation}{subsection}
\counterwithout{subsection}{section}
\counterwithin{figure}{subsection}

\newtheorem{theorem}[equation]{Theorem}
\newtheorem*{theorem*}{Theorem}
\newtheorem{lemma}[equation]{Lemma}
\newtheorem*{lemma*}{Lemma}
\newtheorem{corollary}[equation]{Corollary}
\newtheorem{proposition}[equation]{Proposition}
\newtheorem*{proposition*}{Proposition}

\theoremstyle{definition}
\newtheorem{definition}[equation]{Definition}
\newtheorem*{definition*}{Definition}
\newtheorem{remark}[equation]{Remark}
\newtheorem*{remark*}{Remark}

\newtheorem{question}[equation]{Question}

\newtheorem{example}[equation]{Example}
\newtheorem*{example*}{Example}
\newtheorem*{problem*}{Problem}

\theoremstyle{plain}

\newcounter{intro}

\theoremstyle{plain}
\newtheorem{intro-conjecture}[intro]{Conjecture}
\newtheorem{intro-corollary}[intro]{Corollary}
\newtheorem{intro-theorem}[intro]{Theorem}
\newtheorem{intro-proposition}[intro]{Proposition}
\theoremstyle{definition}
\newtheorem{intro-definition}[intro]{Definition}
\newtheorem{intro-example}[intro]{Example}
\newtheorem{intro-remark}[intro]{Remark}
\newtheorem{intro-question}[intro]{Question}

\newcommand{\theoremref}[1]{\hyperref[#1]{Theorem~\ref*{#1}}}
\newcommand{\lemmaref}[1]{\hyperref[#1]{Lemma~\ref*{#1}}}
\newcommand{\definitionref}[1]{\hyperref[#1]{Definition~\ref*{#1}}}
\newcommand{\propositionref}[1]{\hyperref[#1]{Proposition~\ref*{#1}}}
\newcommand{\conjectureref}[1]{\hyperref[#1]{Conjecture~\ref*{#1}}}
\newcommand{\corollaryref}[1]{\hyperref[#1]{Corollary~\ref*{#1}}}
\newcommand{\exampleref}[1]{\hyperref[#1]{Example~\ref*{#1}}}

\DeclareMathOperator{\gr}{gr}

\DeclareMathOperator{\dr}{DR}

\DeclareMathOperator{\lcdef}{lcdef}
\DeclareMathOperator{\depth}{depth}

\DeclareMathOperator{\codim}{codim}

\DeclareMathOperator{\pt}{pt}

\DeclareMathOperator{\Sch}{Sch}
\DeclareMathOperator{\Sm}{Sm}
\DeclareMathOperator{\Spc}{Spc}

\makeatletter
\let\old@caption\caption
\renewcommand*{\caption}[1]{%
	\setcounter{figure}{\value{equation}}%
	\stepcounter{equation}%
	\old@caption{#1}\relax%
}
\makeatother

\begin{document}

\title{The Dual Complex and RC-Invariant Functors}

\author{Andrew Burke}
\address{Department of Mathematics, Harvard University, 1 Oxford Street, Cambridge, MA 02138, USA}
\email{aburke@math.harvard.edu}


\date{\today}

\begin{abstract}
We extend the dual complex of a simple normal crossings divisor to a homotopy type invariant $\mathcal{D}(X)$ of arbitrary complex varieties, and we prove that it is contractible when $X$ is connected and of klt type. In fact, the dual complex belongs to a broader class of invariants, which we call RC (rationally connected)-invariant functors. We prove a rigidity theorem governing the behavior of such functors with respect to rationally connected fibrations.
\end{abstract}

\maketitle

\makeatletter
\newcommand\@dotsep{4.5}
\def\@tocline#1#2#3#4#5#6#7{\relax
  \ifnum #1>\c@tocdepth 
  \else
    \par \addpenalty\@secpenalty\addvspace{#2}%
    \begingroup \hyphenpenalty\@M
    \@ifempty{#4}{%
      \@tempdima\csname r@tocindent\number#1\endcsname\relax
    }{%
      \@tempdima#4\relax
    }%
    \parindent\z@ \leftskip#3\relax
    \advance\leftskip\@tempdima\relax
    \rightskip\@pnumwidth plus1em \parfillskip-\@pnumwidth
    #5\leavevmode\hskip-\@tempdima #6\relax
    \leaders\hbox{$\m@th
      \mkern \@dotsep mu\hbox{.}\mkern \@dotsep mu$}\hfill
    \hbox to\@pnumwidth{\@tocpagenum{#7}}\par
    \nobreak
    \endgroup
  \fi}
\def\l@section{\@tocline{1}{0pt}{1pc}{}{\bfseries}}
\def\l@subsection{\@tocline{2}{0pt}{25pt}{5pc}{}}
\makeatother
\tableofcontents

\section{Introduction}

\noindent
{\bf Dual Complex.}
The \emph{dual complex} of a simple normal crossings (snc) divisor $E = \cup E_i$ is a cell complex $\mathcal{D}(E)$ encoding the incidence relations among the irreducible components of $E$. Its vertices correspond to the irreducible components $E_i$, its edges to the connected components of pairwise intersections $E_i \cap E_j$, and more generally its $p$-cells to the connected components of $(p+1)$-fold intersections $E_{i_0} \cap \ldots \cap E_{i_p}$.

The topology of $\mathcal{D}(E)$ is studied in \cite{stepanov, thu, abw, payne, dFKX}, especially when $E$ is the exceptional divisor of a log resolution of singularities. In this case, the homotopy type of $\mathcal{D}(E)$ is a singularity invariant. We start by recording two seminal results from \cite{dFKX} below. All varieties are assumed to be quasi-projective, separated, and of finite type over $\mathbb{C}$. Recall that such a variety $X$ is of \emph{klt type} if there exists a boundary $\mathbb{Q}$-divisor $\Delta$ such that the pair $(X, \Delta)$ is klt. 

\begin{intro-theorem}[{\cite[Theorems 2 and 4]{dFKX}}]\label{intro-thm dFKX}
The following statements hold.
\begin{enumerate}
    \item Let $x \in X$ be a closed point of a variety $X$ of klt type. If $f: Y \rightarrow X$ is a resolution such that $E = f^{-1}(x)_{\mathrm{red}}$ is an snc divisor, then $\mathcal{D}(E)$ is contractible.
    \item Let $f: Y \rightarrow C$ be a dominant projective morphism from a smooth connected variety to a smooth curve whose general fiber is rationally connected. If $x \in C$ is a closed point and $E = f^{-1}(x)_{\mathrm{red}}$ is an snc divisor, then $\mathcal{D}(E)$ is contractible.
\end{enumerate}
\end{intro-theorem}

In this paper, we extend the dual complex construction to arbitrary varieties. More precisely, to every variety $X$ we associate a cell complex $\mathcal{D}(X)$, well-defined up to homotopy equivalence, which agrees with the usual dual complex when $X$ is an snc divisor. We call $\mathcal{D}(X)$ the (generalized) dual complex of $X$. 

The formal definition uses hyperresolutions and is given in \S \ref{subsection exmps}. However, Corollary \ref{cor X subset A} shows that $\mathcal{D}(X)$ can be computed from an embedded resolution. Indeed, suppose that $i: X \hookrightarrow V$ is a closed embedding into a smooth variety and $g: W \rightarrow V$ is a projective birational morphism, with $W$ smooth and $E = g^{-1}(X)_{\mathrm{red}}$ an snc divisor. Then there is a homotopy equivalence
\begin{align}\label{eqn dual resn}
    \mathcal{D}(X) \simeq \mathcal{D}(E),
\end{align}
where $\mathcal{D}(E)$ denotes the usual dual complex of $E$. This formula is a topological analogue of the embedded resolution description of the zeroth Du Bois complex in \cite{schwede}; see also \cite{steenbrink}.

Our main result concerning the dual complex is the following rigidity theorem.

\begin{intro-theorem}\label{intro-thm dual complex}
Suppose $f: Y \rightarrow X$ is a dominant projective morphism between connected varieties of klt type whose general fiber is rationally connected. Then for any closed subvariety $Z \subset X$, there is a homotopy equivalence
\begin{align*}
    \mathcal{D} \big( f^{-1}(Z) \big) \simeq \mathcal{D} \big( Z \big).
\end{align*}
\end{intro-theorem}

Taking $Z = \{ x \}$ recovers both parts of Theorem \ref{intro-thm dFKX} as special cases. On the other hand, applying Theorem \ref{intro-thm dual complex} with $Z = X$ and with $f: Y \rightarrow X$ a resolution of singularities gives the following statement.

\begin{intro-corollary}\label{intro-cor klt}
If $X$ is connected and of klt type, then its dual complex
$\mathcal{D}(X)$ is contractible.
\end{intro-corollary}

The proof of Theorem \ref{intro-thm dual complex} analyzes a series of carefully chosen minimal model programs using the results of \cite{bchm}, and it is inspired by arguments in \cite{dFKX}. The flexibility of our generalized dual complex streamlines the argument and permits a uniform statement.

The generalized dual complex also has a Hodge-theoretic interpretation, extending the well-known case of snc divisors. 

\begin{intro-proposition}\label{intro-prop weight zero}
Let $X$ be a projective variety. For every $i \geq 0$, there is an isomorphism
\begin{align*}
    H^i(\mathcal{D}(X), \mathbb{Q}) \simeq W_0 H^i(X, \mathbb{Q}).
\end{align*}
Consequently:
\begin{enumerate}
    \item If $X$ has rational singularities, then $\mathcal{D}(X)$ is $\mathbb{Q}$-contractible.
    \item If $X$ is equidimensional and $s = \dim \mathrm{Sing}(X)$, then $H^i( \mathcal{D}(X), \mathbb{Q}) = 0$ for all $0 < i < \mathrm{depth} \mathscr{O}_X - s$. In particular, if $X$ has isolated Cohen-Macaulay singularities, then $\mathcal{D}(X)$ is a $\mathbb{Q}$-homology wedge of spheres.
\end{enumerate}
\end{intro-proposition}

By a theorem of Elkik, every variety of klt type has rational singularities. When $X$ is connected and of klt type, we have seen that $\mathcal{D}(X)$ is contractible. The proposition extends a weaker form of this statement to rational singularities and $\mathbb{Q}$-contractibility. Moreover, it is sharp. Indeed, Remark \ref{rmk stepanov} discusses an example, due to \cite{payne}, of a threefold $X$ with rational singularities for which $\mathcal{D}(X) \simeq \Sigma \mathbb{RP}^2$.

Part $(2)$ of the proposition follows from a Lefschetz hyperplane theorem for singular varieties, Theorem \ref{thm lefschetz}, proved in the Appendix. In fact, Theorem \ref{thm lefschetz} gives a stronger vanishing, where we replace $\depth \mathscr{O}_X$ with $\depth \underline{\Omega}_X^0$ for the zeroth Du Bois complex $\underline{\Omega}_X^0$ of $X$. Theorem \ref{thm lefschetz} refines the Lefschetz theorems of \cite[Theorem~E and Theorem~9.2]{parkpopa} according to the depth of Du Bois complexes, and it may be of independent interest.

\noindent
{\bf RC-Invariant Functors.}
Our proof of Theorem \ref{intro-thm dual complex} relies only on a formal invariance property of the dual complex. It therefore applies not only to the dual complex, but also to a more general class of invariants. These invariants are characterized by their compatibility with rational connectedness, and they form the main objects of study in this paper.

\begin{intro-definition}\label{intro-defn P^1 invt}
An \emph{RC-invariant functor} is a functorial assignment $X \mapsto F(X)$ from non-empty, connected, smooth varieties to sets or modules such that, for every dominant projective morphism $f: Y \rightarrow X$ whose general fiber is rationally connected, the induced map 
\begin{align*}
    F(f): F(Y) \rightarrow F(X)
\end{align*} 
is an isomorphism.
\end{intro-definition}

Equivalently, $F(f)$ is an isomorphism whenever $f$ is a model for an MRC fibration. Thus, $F$ is insensitive to rationally connected directions. In particular, if $X$ is rationally connected, then $F(X) \simeq F( \operatorname{Spec}\mathbb{C} )$.

Every RC-invariant functor admits an extension $G = RF$ to arbitrary singular varieties, valued in homotopy equivalence classes of spaces or of chain complexes, respectively; see \S \ref{subsection properties} for details. If $\epsilon_\bullet: X_\bullet \rightarrow X$ is a hyperresolution of $X$, then the extension is defined by the homotopy colimit
\begin{align*}
    G \big( X \big) :=  \operatorname*{hocolim} F \big( X_\bullet \big).
\end{align*}
Alternatively, $G(X)$ can be computed by a formula analogous to \eqref{eqn dual resn}; see Corollary \ref{cor X subset A}.

We actually define RC-invariant functors with values in an arbitrary cocomplete $\infty$-category. However, all of our examples except for the fundamental group $\pi_1(X)$ and the structure sheaf $\mathscr{O}_X$ fit into the format of Definition \ref{intro-defn P^1 invt}, so we ignore this point until the body of the paper.

With this terminology, the generalized dual complex is the extension $\mathcal{D} = R h^0$ of the RC-invariant functor $h^0$ that sends every non-empty, connected, smooth variety to a singleton. Many familiar invariants of varieties define RC-invariant functors.

\begin{intro-example}\label{intro-exmp RC-functor}
The following are RC-invariant functors: $h^0(X)$, $H^0(X, \mathbb{Z})$, $H^1(X, \mathbb{Z})$, $\pi_1(X)$, $CH_0(X)$, $H^j(X, \mathcal{O}_X)$, $\mathcal{O}_X$, $F^2 CH^2(X)_\mathbb{Q}$, and $X(\mathbb{C})/R$. 
\end{intro-example}

The extensions of these familiar RC-invariant functors to arbitrary singular varieties are often related to the more naive extensions by the same formula. For example, 
\[
    H_0\bigl(RCH_0(X)\bigr)\simeq CH_0(X),
    \qquad
    R\mathscr{O}_X\simeq \underline{\Omega}_X^0,
    \qquad\text{and}\qquad
    \pi_1\bigl(R\tau_{\leq 1}(X)\bigr)\simeq \pi_1(X).
\]
The use of the $1$-truncated analytic space in the final comparison is essential for obtaining a clean relationship with fundamental group. See \S \ref{subsection exmps} for precise statements and for further comparisons.

As mentioned earlier, there is a generalization of Theorem \ref{intro-thm dual complex} to RC-invariant functors, and this is our main theorem.

\begin{intro-theorem}\label{intro-theorem main}
Let $F$ be an RC-invariant functor with extension $G = RF$. Suppose $f: Y \rightarrow X$ is a dominant projective morphism between connected varieties of klt type whose general fiber is rationally connected. Then for any closed subvariety $Z \subset X$, there is an equivalence 
\begin{align*}
    G \big( f^{-1}(Z) \big) \simeq G \big( Z \big).
\end{align*}
\end{intro-theorem}

Thus, Theorem~\ref{intro-thm dual complex} follows from Theorem~\ref{intro-theorem main} by taking $F = h^0$.

Let us record several geometric consequences of Theorem \ref{intro-theorem main} arising from the RC-invariant functors listed in Example \ref{intro-exmp RC-functor}. We begin with $R$-equivalence, introduced in \cite{manin}. For a complex variety $X$, let $X(\mathbb{C})/R$ denote the set of \emph{$R$-equivalence classes} of closed points, where two points are equivalent if they are joined by a chain of rational curves. This defines an RC-invariant functor, and the connected components of its extension space are naturally identified with $X(\mathbb{C})/R$ in Proposition \ref{prop R-equ}. Thus, it is connected if and only if $X$ is rationally chain connected. Applying Theorem \ref{intro-theorem main} to this functor gives statement $(1)$ of the corollary below, while the RC-invariant functors $CH_0(X)$, $\mathscr{O}_X$, and $\pi_1(X)$ yield statements $(2)$, $(3)$, and $(4)$, respectively.

\begin{intro-corollary}\label{intro-cor fibers}
Let $f: Y \rightarrow X$ be a dominant projective morphism between connected varieties of klt type whose general fiber is rationally connected. Then every closed fiber $F$ of $f$ has the following properties:
\begin{enumerate}
    \item $F$ is rationally chain connected;
    \item $CH_0(F) \simeq \mathbb{Z}$;
    \item $\mathbb{H}^i(F, \underline{\Omega}_F^0) = 0$ for every $i > 0$; and
    \item $F$ is simply connected.
\end{enumerate}
\end{intro-corollary}

A more general statement concerning inverse images of arbitrary closed subvarieties appears in Corollary \ref{cor preimages}

Statement $(1)$ generalizes \cite[Corollary $1.5$]{hm rat conn}, which establishes rational chain connectedness for the fibers of a resolution of a klt type variety. That result is also an ingredient in the proof of Theorem \ref{intro-theorem main}. When $f$ is equidimensional, $(1)$ follows from \cite[Corollary IV.$3.5$]{rational curves}. The non-equidimensional case appears to be new, even for smooth varieties.

In forthcoming work \cite{burke CH_0}, we show that statements $(2)$ and $(3)$ remain valid under weaker hypotheses on the general fiber. Namely, $(2)$ holds assuming only that the general fiber satisfies $CH_0(F) \simeq \mathbb{Z}$, without requiring rational chain connectedness. Similarly, $(3)$ holds if the general fiber is connected and satisfies $\mathbb{H}^i(F, \underline{\Omega}_F^0) = 0$ for every $i > 0$. By contrast, no analogous statement holds for the fundamental group; see Example \ref{exmp pi_1}.

\noindent
{\bf Geometry of strata.}
Our main Theorem \ref{intro-theorem main} is well-suited to studying global invariants of fibers. However, some applications require finer information about their individual components and strata. In \S \ref{section collapse}, we refine our arguments to capture such information. As an illustration, consider the following question of Koll\'ar.

\begin{intro-question}[{\cite[Question $16$]{ax}}]\label{intro-question}
Let $f: Y \rightarrow C$ be a dominant projective morphism from a smooth connected variety to a smooth curve whose general fiber is rationally connected. If $x \in C$ is a closed point and $E = f^{-1}(x)_{\mathrm{red}}$ is an snc divisor, must $E$ have a rationally connected irreducible component?
\end{intro-question}

In this setting, $E$ is rationally chain connected and its dual complex is contractible. However, these are both global statements, and neither implies that an individual component of $E$ is rationally connected. Hogadi and Xu answered Question \ref{intro-question} affirmatively when $\dim Y \leq 4$ \cite{hogadi xu}, and Brown and Foster proved it in general \cite{brown foster}. We go further by describing the rational connectedness
behavior of all strata of $E$.

\begin{intro-theorem}\label{intro-thm degeneration}
In the setting of Question \ref{intro-question}, the divisor $E$ contains a rationally connected irreducible component $E_0$.

Moreover, the strata of $E$ other than $E_0$ admit an ordered partition into pairs
\[
\big( Z_1,W_1 \big ),\ldots, \big( Z_a, W_a \big)
\]
such that, for every \(i\),
\begin{enumerate}
    \item \(Z_i\) is a divisor in \(W_i\) with $f(Z_i) = f(W_i)$;
    \item $Z_i$ and $W_i$ have birational MRC bases; and
    \item every stratum of \(E\) contained in \(W_i\) is equal to \(Z_j\) or \(W_j\) for some \(j\leq i\).
\end{enumerate}
\end{intro-theorem}

The ordered partition in Theorem \ref{intro-thm degeneration} has a combinatorial interpretation in terms of the dual complex $\mathcal{D}(E)$. Namely, each pair $(Z_i, W_i)$ determines an elementary collapse. The ordering condition ensures that these collapses can be performed successively until only the vertex corresponding to $E_0$ remains. Thus, $\mathcal{D}(E)$ is \emph{collapsible}. This is strictly stronger than the contractibility of $\mathcal{D}(E)$. Our proof is inspired by the study of collapsibility in \cite{dFKX}. The additional geometric content is that the collapse preserves MRC bases. Namely, the two strata in each pair have birational MRC bases, and the final stratum $E_0$ is rationally connected.

Theorem \ref{intro-thm degeneration} is a special case of the more general Theorem \ref{thm cartier collapse}. The latter also has consequences for resolutions of klt singularities.

\begin{intro-theorem}\label{intro-thm collapse}
Let $X$ be a variety of klt type and $Z \subset X$ an irreducible subvariety. Then there exists a resolution $f: Y \rightarrow X$ such that $E = f^{-1}(Z)_{\mathrm{red}}$ is an snc divisor and the following holds. There exists an irreducible component $E_0$ of $E$ such that the map $E_0 \rightarrow Z$ is dominant with rationally connected general fibers, and the remaining strata of $E$ that dominate $Z$ admit an ordered partition into pairs satisfying conditions \textup{(1)}--\textup{(3)} of Theorem \ref{intro-thm degeneration}.

If $X$ is $\mathbb{Q}$-factorial, these properties hold for every resolution $f: Y \rightarrow X$ such that $f^{-1}(Z)_{\mathrm{red}}$ is an snc divisor.
\end{intro-theorem}

When $X$ is not $\mathbb{Q}$-factorial, it remains unclear whether every such resolution satisfies these properties.

\noindent
{\bf Integral weights and negative $K$-theory.}
We conclude with an application of Theorem \ref{intro-theorem main} to the weight filtration on cohomology. Deligne endowed the rational cohomology of every complex variety with a weight filtration, and Gillet-Soulé \cite{GS} constructed an integral refinement. For a projective variety $X$, this can be described explicitly using a hyperresolution. Indeed, a hyperresolution $\epsilon_\bullet: X_\bullet \rightarrow X$ induces a \emph{weight spectral sequence}
\begin{align}\label{intro-eqn weight ss}
    E_1^{p,q} = H^q(X_p, \mathbb{Z}) \Rightarrow H^{p+q}(X, \mathbb{Z}).
\end{align}

After tensoring with $\mathbb{Q}$, the spectral sequence degenerates at the $E_2$-page, and the resulting filtration on $H^\bullet(X,\mathbb Q)$ is Deligne's weight filtration. Over $\mathbb{Z}$, the absence of an integral Hodge filtration means that degeneration at $E_2$ need not occur.  Nevertheless, the spectral sequence defines an integral weight filtration on
\(H^\bullet(X,\mathbb Z)\), with $\gr_q^W H^{p+q}(X, \mathbb{Z})$ a subquotient of $E_2^{p,q}$.

The two lowest rows of the $E_2$-page are computed by the extension cochain complexes of the RC-invariant functors $H^0(X,\mathbb{Z})$ and $H^1(X, \mathbb{Z})$:
\begin{align}\label{intro-eqn E_2}
    E_2^{i,0} \simeq H^i \big( RH^0(X, \mathbb{Z}) \big), \text{ and } E_2^{i,1} \simeq H^i\big( RH^1(X, \mathbb{Z}) \big). 
\end{align}
In fact, the first of these groups is isomorphic to $H^i ( \mathcal{D}(X), \mathbb{Z} )$, giving an integral refinement of Proposition \ref{intro-prop weight zero}. Applying Theorem \ref{intro-theorem main} with $F = H^0$ and $F = H^1$ yields the following restriction on the integral weights for varieties of klt type. 

\begin{intro-corollary}\label{intro-cor integral weight}
Let $X$ be a projective variety of klt type. Then, for every $i \geq 0$,
\[
    H^i(X, \mathbb{Z}) \text{ has weights } \geq \min \{ 2, i\}.
\]
\end{intro-corollary}

It is instructive to compare Corollary \ref{intro-cor integral weight} with the situation for rational singularities. If $X$ is projective and has rational singularities, then by Lemma \ref{lem weight restrictions},
\[
    H^i(X, \mathbb{Q}) \text{ has weights } \geq \min \{ 2, i\}.
\]
Thus, Corollary \ref{intro-cor integral weight} gives an integral refinement of this rational weight vanishing for klt type singularities. 

In joint work with Popa and Shen \cite{bps}, we apply this circle of ideas to the negative $K$-groups of singular varieties. For example, we identify the bottom negative $K$-group of an $n$-dimensional variety $X$ with a term in the weight spectral sequence \ref{intro-eqn weight ss}:
\begin{align*}
    K_{-n}(X) \simeq E_2^{n,0}
\end{align*}
Combining Theorem \ref{intro-theorem main} with \ref{intro-eqn E_2}, we conclude that $K_{-n}(X) = 0$ whenever $X$ is of klt type; see \cite[Theorem $10.1$]{bps}.

The finer analysis of strata underlying Theorem \ref{intro-thm collapse} also provides information about $K_{-n+2}$. More precisely, we use an ad hoc version of that theorem to show that, if $X$ is projective of dimension $n \geq 3$ with isolated singularities of klt type, then
\begin{align*}
    K_{-n+2}(X)_\mathbb{Q} \simeq W_2 H^n(X, \mathbb{Q}).
\end{align*}
See \cite[Theorems $9.9$ and $12.1$]{bps}. Together, these results fit into a broader conjectural framework in which the weight filtration on cohomology governs negative $K$-groups.

\noindent
{\bf Acknowledgements.}
I would like to thank Rob Lazarsfeld, Sung Gi Park, Mihnea Popa, Wanchun Shen, Anda Tenie, Anh Duc Vo, and Chenyang Xu for helpful conversations regarding the contents of this paper.

\section{Preliminaries}

\subsection{Hyperresolutions and extension to singular varieties}

In this section, we review categorical background and state an extension criterion, Theorem \ref{thm extension} from \cite{infinity cat}, which is inspired by the work of \cite{guillen navarro}. The theorem allows us to extend suitable functors defined on smooth varieties to singular varieties. We will later apply it to RC-invariant functors, defined in \S \ref{subsection properties}.

Throughout, all schemes are separated and of finite type over $\mathbb{C}$ and all varieties are quasi-projective over $\mathbb{C}$. Let $\mathrm{Sch}$ denote the category of such schemes and let $\mathrm{Sm} \subset \mathrm{Sch}$ denote the full subcategory of smooth, connected, non-empty varieties. We also require the notion of an abstract blow-up square.

\begin{definition}\label{defn abstract blow up square}
An \emph{abstract blow-up square} is a Cartesian diagram
\begin{equation}\label{eqn abus}
\begin{tikzcd}
E \arrow[r] \arrow[d]
  & Y \arrow[d, "f"] \\
Z \arrow[r, "i"]
  & X,
\end{tikzcd}
\end{equation}
such that $f$ is proper, $i$ is a closed embedding, and $f$ induces an isomorphism $Y \setminus E \simeq X \setminus Z$. We say that \eqref{eqn abus} is a \emph{smooth blow-up square} if $X$ is a smooth variety, $Z \subset X$ is a smooth subvariety, and $Y$ is the blow-up of $X$ along $Z$.
\end{definition}

We may now state the extension theorem using the language of $\infty$-categories. The relevant examples for our applications are the $\infty$-categories $\mathrm{Spc}$ of topological spaces and $\mathcal{D}(R)_{\geq 0}$ of non-negatively graded complexes of $R$-modules, where $R$ is a ring (or a sheaf of rings). Readers unfamiliar with $\infty$-categories may, for intuition, think instead of the corresponding homotopy categories.

\begin{theorem}[{\cite[Theorem $2.3.5$]{infinity cat}}]\label{thm extension}
Let $F: \Sm \rightarrow \mathcal{D}$ be a functor to a cocomplete $\infty$-category $\mathcal{D}$. Suppose that for every smooth blow-up square 
\[
\begin{tikzcd}
E \arrow[r] \arrow[d]
  & Y \arrow[d] \\
Z \arrow[r]
  & X,
\end{tikzcd}
\]
the induced diagram
\[
\begin{tikzcd}
F(E) \arrow[r] \arrow[d]
  & F(Y) \arrow[d] \\
F(Z) \arrow[r]
  & F(X)
\end{tikzcd}
\]
is a pushout square in $\mathcal{D}$.

Then $F$ admits a unique extension $G = RF: \Sch \rightarrow \mathcal{D}$, such that $G(\emptyset)$ is the initial object and every abstract blow-up square is sent to a pushout square in $\mathcal{D}$.

More explicitly, if $X$ is a singular variety and $\epsilon_\bullet: X_\bullet \rightarrow X$ is a cubical hyperresolution, then
\begin{align}\label{eqn colim}
    G(X) \simeq \operatorname*{colim} F(X_\bullet).
\end{align}
\end{theorem}

For example, $G$ is insensitive to non-reduced structure and $G$ commutes with coproducts:
\begin{align}\label{eqn coprod}
    G \big( \coprod X_i\big) \simeq \coprod G \big( X_i \big).
\end{align}

We will recall the notion of a cubical hyperresolution below. For now, let us emphasize that the colimit in \eqref{eqn colim} is a homotopy colimit. We illustrate this point in the $\infty$-category $\mathrm{Spc}$ of spaces by comparing homotopy pushouts with ordinary topological pushouts.

\begin{remark}[{\bf Homotopy pushout}]\label{rmk homotopy pushout}
Suppose that $f: C \rightarrow A$ and $g: C \rightarrow B$ are continuous maps. The \emph{topological pushout} is the quotient space $(A \coprod B)/\sim$, where $f(c) \sim g(c)$ for every $c \in C$. This construction is not homotopy invariant. For example, consider the diagrams
\begin{equation}\label{eqn top pushout}
\begin{tikzcd}
\pt \coprod \pt \arrow[r] \arrow[d] & \pt  \\
\pt  & 
\end{tikzcd}
\qquad\qquad
\begin{tikzcd}
\pt \coprod \pt \arrow[r, "g"] \arrow[d] & I \\
\pt  & 
\end{tikzcd}
\end{equation}
where $I = [0,1]$ and $g$ sends the two points to the endpoints $0$ and $1$. These diagrams are objectwise homotopy equivalent, but their topological pushouts are not. Indeed, the first pushout is a point, whereas the second pushout is the circle $S^1$.

The homotopy pushout fixes this failure. A model for the \emph{homotopy pushout} is the quotient space 
\begin{align*}
    \bigl(A\coprod(C\times I)\coprod B\bigr)\big/\sim
\end{align*}
with relations $(c, 0) \sim f(c)$ and $(c,1) \sim g(c)$ for every $c \in C$. The homotopy pushout of both diagrams in \eqref{eqn top pushout} is $S^1$.
\end{remark}

Let us briefly recall the notions of cubical and semi-simplicial hyperresolutions. Lemma \ref{lem colimit simplicial} will show that the colimit in \eqref{eqn colim} may be computed using either type of hyperresolution. 

To describe the indexing category, for an integer $n \geq 0$, let $\square_n$ denote the category whose objects are the non-empty subsets $J \subset [n] := \{ 0 , \ldots, n\}$, with a unique morphism $J \rightarrow J'$ whenever $J \subset J'$. A \emph{cubical hyperresolution} of a variety $X$ consists of a functor
\[
    X_\bullet: \square_n^{\mathrm{op}} \rightarrow \mathrm{Sm}
\]
for some $n \geq 0$, together with compatible morphisms $\epsilon_\bullet: X_\bullet \rightarrow X$ of each piece $X_J$ to $X$. This data is required to arise from an iterated system of abstract blow-up squares resolving $X$. We refer to \cite{gnpp} and \cite[Chapter $5$]{peters steenbrink} for the precise definition and for an algorithm to construct a cubical hyperresolution. Every variety $X$ admits a cubical hyperresolution satisfying the dimension bound $\dim X_J \leq \dim X - |J|+1$ for every $J \subset [n]$.

A cubical hyperresolution $\epsilon_\bullet: X_\bullet \rightarrow X$ canonically gives rise to a semi-simplicial hyperresolution. To describe ths construction, let $\Delta_n$ denote the category whose objects are the ordered sets $[k] = \{ 0, \ldots , k \}$ for $0 \leq k \leq n$, and whose morphisms are strictly increasing maps. A \emph{semi-simplicial hyperresolution} of a variety $X$ consists of a functor 
\[
    Y_\bullet: \Delta_n^{\mathrm{op}} \rightarrow \Sm,
\]
for some $n \geq 0$, together with compatible morphisms $\delta_\bullet: Y_\bullet \rightarrow X$ of each piece $Y_k$ to $X$.

The semi-simplicial hyperresolution associated to $\epsilon_\bullet$ has terms
\[
Y_k=\coprod_{|J|=k+1} X_J,
\qquad 0\leq k\leq n.
\]
See \cite[Remark $5.4$]{peters steenbrink} for details. From a categorical perspective, $Y_\bullet$ is the left Kan extension of $X_\bullet$ along the natural functor $u: \square_n^{\mathrm{op}} \rightarrow \Delta_n^{\mathrm{op}}$ sending $J$ to $[|J|-1]$. This latter description is useful for the proof of Lemma \ref{lem colimit simplicial}.

\begin{example}[{\bf Snc hyperresolution}]\label{exmp snc hyperres}
An snc divisor $E = \bigcup_{i=0}^n E_i$ carries a standard cubical hyperresolution whose terms are the strata $E_J := \bigcap_{j \in J} E_j$ for all $J \subset [n]$.

If we set $E^{[k]} := \coprod_{|J| = k+1} E_J$, then the associated semi-simplicial resolution has the form
\begin{align*}
     \bigg(  \ldots E^{[2]}
    \mathrel{\substack{\rightarrow\\[-0.6ex]\rightarrow\\[-0.6ex]\rightarrow}}
    E^{[1]} \rightrightarrows E^{[0]} \bigg) \rightarrow E.
\end{align*}
\end{example}

The following lemma shows that the colimit appearing in \eqref{eqn colim} may be computed either from a cubical hyperresolution or from its associated semi-simplicial hyperresolution. This flexibility is useful in applications.

\begin{lemma}\label{lem colimit simplicial}
Let $\epsilon_\bullet: X_\bullet \rightarrow X$ be a cubical hyperresolution, with associated semi-simplicial hyperresolution $\delta_\bullet: Y_\bullet \rightarrow X$. If $F: \mathrm{Sm} \rightarrow \mathcal{D}$ is a functor to a cocomplete $\infty$-category $\mathcal{D}$, then there is a natural equivalence
\[
    \operatorname*{colim}_{\square_n^{\mathrm{op}}} F(X_\bullet)
    \simeq
    \operatorname*{colim}_{\Delta_n^{\mathrm{op}}} F(Y_\bullet).
\]
\end{lemma}

\begin{proof}
One can prove this by a direct computation. Alternatively, it follows formally from the functoriality of left Kan extensions. Indeed, consider the diagram
\[
\begin{tikzcd}[column sep=4.5em, row sep=3em]
\Box_n^{\mathrm{op}}
  \arrow[r, "X_\bullet"]
  \arrow[d, "u"']
  \arrow[dd, bend right=55, "p"']
&
\mathrm{Sm}
  \arrow[r, "F"]
&
\mathcal{D}
\\
\Delta_n^{\mathrm{op}}
  \arrow[ur, "Y_\bullet"']
  \arrow[d, "q"']
\\
*
\end{tikzcd}
\]
The colimit $\operatorname*{colim} F(X_\bullet)$ is the left Kan extension of $F \circ X_\bullet$ along the functor $p$ to the terminal category $*$ by \cite[Proposition $6.5.1$]{riehl}. But $p$ factors as $q \circ u$. The left Kan extension of $F \circ X_\bullet$ along $u$ is $F \circ Y_\bullet$, and the left Kan extension of $F \circ Y_\bullet$ along $q$ is the colimit $\operatorname*{colim} F(Y_\bullet)$, again by \cite[Proposition $6.5.1$]{riehl}. Our claim therefore follows from the composition rule for left Kan extensions \cite[Exercise $6.1$iii]{riehl}.
\end{proof}

\subsection{Integral weight filtrations}

In this section, we discuss the weight filtration defined on the integral cohomology of a projective variety introduced by \cite{GS}. We follow the treatment of \cite[\S 2]{bps}.

Let $X$ be a projective variety, and choose a cubical hyperresolution $\epsilon_\bullet: X_\bullet \rightarrow X$. By cohomological descent, the constant sheaf $\mathbb{Z}_X$ is quasi-isomorphic to the total complex associated to the derived pushforwards
\begin{align}\label{eqn hyperres cone}
    \mathbb{Z}_X \simeq \bigg( R\epsilon_{0*} \mathbb{Z}_{X_0} \rightarrow R\epsilon_{1*} \mathbb{Z}_{X_1} \rightarrow \ldots \bigg).
\end{align}

Let us first review the situation after tensoring with $\mathbb{Q}$. We obtain from \eqref{eqn hyperres cone} a \emph{weight spectral sequence}
\begin{align}\label{eqn ratl weight ss}
    E_1^{p,q} = H^q(X_p, \mathbb{Q}) \Rightarrow H^{p+q}(X, \mathbb{Q}).
\end{align}
Since $X_p$ is smooth and projective, the terms $E_1^{p,q}$ naturally carry pure Hodge structures of weight $q$, and the boundary maps $d_1$ respect the Hodge filtration. It follows that the subquotients $E_2^{p,q}$ are pure of weight $q$ as well. The boundary maps $d_2: E_2^{p,q} \rightarrow E_2^{p+2, q-1}$ are between pure Hodge structures of different weight and they preserve the Hodge filtration. Therefore, they all vanish. The same argument shows the boundary maps $d_r$ for $r \geq 2$ vanish as well. This shows that the spectral sequence degenerates at the $E_2$ page. The filtration induced on $H^{p+q}(X, \mathbb{Q})$ is, by definition, Deligne's weight filtration.

Let us work now with integral coefficients. The iterated cone \eqref{eqn hyperres cone} again gives rise to a cohomology spectral sequence
\begin{align}\label{eqn weight ss}
    E_1^{p,q} = H^q(X_p, \mathbb{Z}) \Rightarrow H^{p+q}(X, \mathbb{Z}).
\end{align}
This time, given the lack of a Hodge filtration, the spectral sequence need not degenerate. Nevertheless, it is an invariant of $X$.

\begin{proposition}[{\cite[Theorem $3$]{GS}}]\label{prop integral weight}
Let $X$ be a complex projective variety. The spectral sequence \eqref{eqn weight ss} is an invariant of $X$, from the $E_2$-page onwards. It defines a weight filtration $W_\ell H^i(X, \mathbb{Z})$ on the cohomology of $X$ with integral coefficients, which agrees with Deligne's weight filtration after tensoring with $\mathbb{Q}$.
\end{proposition}

\begin{remark}\label{rmk non-projective}
There is an integral weight filtration for arbitrary complex varieties $X$; see \cite[Theorem $3$]{GS}. When $X$ is not projective the construction is less explicit, and we will not use it here.
\end{remark}

We will use the notation
\begin{align*}
    H_q^{p+q}(X, \mathbb{Z}) = E_2^{p,q}.
\end{align*}
The weight piece $\gr_q^W H^{p+q}(X, \mathbb{Z})$ is a subquotient of $H_q^{p+q}(X, \mathbb{Z})$ and these groups coincide after tensoring with $\mathbb{Q}$. 

The weight filtration on integral cohomology carries nontrivial information about the singularities of $X$, even on torsion subgroups, as the following example illustrates.

\begin{example}[{\cite[Proposition $23.2$]{bps}}]\label{example kummer}
Let $A$ be an abelian variety of dimension $g$ and let $X = A/\iota$ be the Kummer variety, where $\iota: x \mapsto -x$ is the negation involution. Then
\[
H^{i}(X,\mathbb{Z}) \;\cong\;
\begin{cases}
\mathbb{Z}^{\binom{2g}{\,i\,}}, 
& \text{if } 0 \le i \le 2g \text{ and } i \text{ is even},\\[6pt]
(\mathbb{Z}/2\mathbb{Z})^{\big(\sum_{j=0}^{2g-i} \binom{2g}{j}\big)}, 
& \text{if } 3 \le i \le 2g-1 \text{ and } i \text{ is odd},\\[6pt]
0, 
& \text{otherwise.}
\end{cases}
\]
When $i$ is even, $H^i(X, \mathbb{Z})$ is pure of weight $i$, and when $i$ is odd, $H^i(X, \mathbb{Z})$ is pure of weight $i-1$.
\end{example}

\subsection{Rational connectedness and MRC fibrations}

We recall the notions of rational connectedness and rational chain connectedness, introduced independently by \cite{kmm} and \cite{campana}, and we record the properties needed below.

\begin{definition}
Let $X$ be a variety. We say that $X$ is \emph{rationally chain connected} if any two points in a dense open subset can be joined by a connected chain of rational curves.

We say that $X$ is \emph{rationally connected} if any two points in a dense open subset lie on a rational curve.
\end{definition}

Rational connectedness is a birational invariant of irreducible varieties, including singular ones. Every rational variety is therefore rationally connected, although the converse is far from true. Nevertheless, rational connectedness is often more tractable than rationality and it is better behaved. 

Moreover, rational connectedness implies rational chain connectedness. For smooth projective varieties, the converse holds by \cite[Theorem $2.1$]{kmm}, but it can fail for singular varieties. For example, a cone over an elliptic curve is rationally chain connected, since any two points can be joined by a chain of two lines passing through the vertex, but it is not rationally connected.

\begin{theorem}[{\cite[Corollaries~$1.5$ and $1.6$]{hm rat conn}}]\label{thm hm}
Let \((X,\Delta)\) be a klt pair.
\begin{enumerate}
    \item If $f: Y \rightarrow X$ is a projective birational morphism, then every fiber of $f$ is rationally chain connected. In particular, if $X$ is rationally chain connected, then $Y$ is rationally chain connected as well.
    \item If $g: X \rightarrow S$ is a dominant projective morphism with connected fibers and $-(K_X + \Delta )$ is $g$-ample, then every fiber of $g$ is rationally chain connected.
\end{enumerate}
\end{theorem}

We will compare varieties using their maximal rationally connected fibrations. This is a construction which extracts the non-rationally-connected part of a variety.

\begin{proposition}[{\cite[Theorem~IV.5.2]{rational curves}}]\label{prop mrc}
Let $X$ be a connected smooth projective variety. There exists a smooth projective variety $Z$, dense open sets $X_0 \subset X$ and $Z_0 \subset Z$, and a dominant rational map $\mu: X \dashrightarrow Z$ such that:
\begin{enumerate}
    \item the restriction $\mu|_{X_0}: X_0 \rightarrow Z_0$ is a proper morphism;
    \item the fibers of $\mu|_{X_0}$ are rationally connected;
    \item for a very general point $z \in Z$, any rational curve in $X$ meeting the fiber $X_z$ is contained in $X_z$; and
    \item $Z$ is not uniruled \cite[Corollary $1.4$]{ghs} .
\end{enumerate}
The rational map $\mu$ is called the \emph{maximal rationally connected fibration} or \emph{MRC fibration} of $X$, and $Z$ is called the \emph{MRC base}. Its birational equivalence class is uniquely determined by the birational type of $X$.
\end{proposition}

A dominant rational map $f:X\dashrightarrow Y$ induces a dominant rational map between MRC bases. If the general fiber of $f$ is rationally connected, the induced map of MRC bases is birational.

More generally, we define the MRC fibration of an irreducible variety to be the MRC fibration of a smooth projective model. This is well-defined up to birational equivalence.

\subsection{Singularities of pairs and MMP}

In this section, we collect terminology and results from the minimal model program (MMP) that will be used in the sequel. Let us begin by reviewing singularities of pairs.
    
\begin{definition}
A \emph{pair} $(X, \Delta)$ consists of a normal variety $X$ and an effective $\mathbb{Q}$-divisor $\Delta$ such that $K_X + \Delta$ is $\mathbb{Q}$-Cartier. If $f: Y \rightarrow X$ is a birational morphism, there is a unique $\mathbb{Q}$-divisor $\Delta_Y$ on $Y$, called the \emph{log pull-back} of $\Delta$, defined by the formulas
\begin{align}\label{eqn delta_Y}
    K_Y + \Delta_Y = f^*(K_X + \Delta), \hspace{.2cm} \text{ and } \hspace{.2cm} f_* \Delta_Y = \Delta. 
\end{align}
If $E$ is a prime divisor on $Y$, the \emph{discrepancy} of $E$ with respect to $(X, \Delta)$, is $a(E, X, \Delta) := - \mathrm{coeff}_E ( \Delta_Y)$.

The pair $(X, \Delta)$ is \emph{lc} (resp. \emph{klt}) if $a(E, X, \Delta) \geq -1$ (resp. $a(E, X, \Delta) > -1)$ for every birational morphism $f: Y \rightarrow X$ and every prime divisor $E$ on $Y$. It is \emph{plt} if it is lc and $a(E, X, \Delta) > -1$ for every prime exceptional divisor $E$ over $X$. An irreducible subvariety $Z \subset X$ is a \emph{lc center} if there exists a prime divisor $E$ over $X$ such that $a(E, X, \Delta)= -1$ and $f(E) = Z$. Thus, an lc pair is klt if and only if it has no lc centers.

The pair $(X, \Delta)$ is an \emph{snc pair} if $X$ is smooth and $\Delta^{=1}$ is an snc divisor. A \emph{log resolution} of $(X, \Delta)$ is a resolution $f: Y \rightarrow X$ of $X$ such that $ \big( Y, \mathrm{Ex}(f) + f_*^{-1}\Delta \big)$ is an snc pair. 

Let $X^{\mathrm{snc}} \subset X$ denote the largest open subset on which $(X,\Delta)$ is an snc pair. An lc pair $(X, \Delta)$ is \emph{dlt} if every lc center meets $X^{\mathrm{snc}}$. In this case, the lc centers of $(X, \Delta)$ are precisely the closures of the strata of $ \Delta^{=1}|_{X^{\mathrm{snc}}}$.
\end{definition}

For the collapse arguments in \S \ref{section collapse}, we will use qdlt (quotient-dlt) pairs, a mild generalization of dlt pairs. This notion was introduced in \cite[Section $5$]{dFKX} in connection with the dual complex and is also used in \cite[Section $2.1$]{xu zhuang}.

\begin{definition}[{\cite[Definition $2.1$]{xu zhuang}}]
Let $E$ be a reduced divisor on a normal variety $Y$. The pair $(Y, E)$ is \emph{toroidal} if, for every point $y \in Y$, there exists an snc pair $y' \in (Y', E')$, a point $y' \in Y'$, and a finite abelian group $G$ acting on $(Y', E')$, such that $G$ preserves every irreducible component of $E'$, acts freely in codimension one, and Zariski locally, there is an isomorphism of pointed pairs
\[
    \big( y \in (Y, E) \big) \cong \big( y' \in (Y',E') \big)/G.
\]
\end{definition}

\begin{definition}[{\cite[Definition $2.2$]{xu zhuang}}]\label{defn qdlt}
An lc pair $(Y, \Gamma)$ is \emph{qdlt} (quotient-dlt) if there exists an open set $U \subset Y$ such that $\Gamma|_U$ is reduced, the pair $(U, \Gamma|_U)$ is toroidal, and every lc center of $(Y, \Gamma)$ meets $U$.
\end{definition}

The lc centers of a qdlt pair $(Y, \Gamma)$ are precisely the connected components of the intersections $E_1 \cap \ldots \cap E_r$, where $E_1, \ldots, E_r$ are irreducible components of $\Gamma^{=1}$ \cite[Proposition $34$]{dFKX}. If $S$ is such an lc center, then $\codim_Y S = r$ and each $E_i$ is $\mathbb{Q}$-Cartier at the generic point of $S$.

Qdlt pairs are preserved by the divisorial contractions and flips of an MMP. Their advantage is that, unlike dlt pairs, they are also preserved by Fano contractions.

\begin{lemma}[{\cite[Proposition $40$]{dFKX}}]\label{lemma qdlt fano}
Let $(Y, \Gamma)$ be a $\mathbb{Q}$-factorial qdlt pair, and let $h: Y \rightarrow W$ be a Fano contraction corresponding to a $(K_Y+\Gamma)$-negative extremal ray. Suppose that $\Gamma^{=1}$ is $h$-vertical. Then $W$ is $\mathbb{Q}$-factorial and there exists a $\mathbb{Q}$-divisor $\Gamma_W$ such that $(W, \Gamma_W)$ is qdlt. Moreover, $h$ induces a bijection between the strata of $\Gamma^{=1}$ and those of $\Gamma_W^{=1}$. 
\end{lemma}

Qdlt pairs are also well-suited to inductive arguments on strata. Indeed,  if $(Y, \Gamma)$ is qdlt and $S$ is an lc center, adjunction endows $S$ with a natural $\mathbb{Q}$-divisor $\Gamma_S = \mathrm{Diff}_S^* \Gamma$ such that $(S, \Gamma_S)$ is qdlt and $K_S + \Gamma_S \sim_\mathbb{Q} \left(K_Y + \Gamma_Y)\right|_S$. This allows us to transfer information about $Y$ to its lc centers.

The following technical lemma will be useful for analyzing klt pairs that are not $\mathbb{Q}$-factorial.

\begin{lemma}\label{lem plt blow-up}
Let $(X, \Delta)$ be a klt pair and let $Z \subset X$ be an irreducible closed subvariety. Then there exists a $\mathbb{Q}$-factorial dlt pair $(Y, E)$, a projective birational morphism $f: Y \rightarrow X$, an open neighborhood $U \subset X$ of the generic point $\eta_Z$, and an irreducible component $E_0$ of $E$ such that $f^{-1}(Z \cap U) = E_0 \cap f^{-1}(U)$.
\end{lemma}

\begin{proof}
Take a plt blow-up of $(X, \Delta)$ over the generic point $\eta_Z$; see \cite[Definition $2.5$ and Remark $2.9$]{moraga plt}. We obtain an open subvariety $U \subset X$ of $\eta_Z$ and a projective birational morphism $g: V \rightarrow U$ such that $g^{-1}(Z \cap U) = F_0$ is an irreducible $\mathbb{Q}$-Cartier divisor and $(V, g_*^{-1}\Delta+F_0)$ is plt. Choose a projective compactification $\overline{g}: \overline{V} \rightarrow X$ of $g$ and extend the boundary on $V$ to a pair $(\overline{V}, \Delta_{\overline{V}})$. Next, take a $\mathbb{Q}$-factorial dlt modification $h: (Y, \Gamma) \rightarrow (\overline{V}, \Delta_{\overline{V}})$ whose restriction over $V$ is small. Set $f = \overline{g} \circ h$ and $E = \Gamma^{=1}$, and let $E_0$ be the strict transform of $F_0$ on $Y$. The preimage $f^{-1}(Z \cap U) = \overline{g}^{-1}(F_0 \cap V)$ is equal to $E_0 \cap f^{-1}(U)$ because $F_0$ is $\mathbb{Q}$-Cartier and $h$ is small on $V$.
\end{proof}

Our arguments rely on running a relative MMP for pairs $(Y, \Gamma)$ over a base $X$. We record two termination results that will be used for MMPs ending with log terminal models and Fano contractions, respectively.

\begin{theorem}[{\cite[Theorem $1.2$]{bchm}}]\label{bchm pseudo-eff}
Let $(Y, \Gamma)$ be a $\mathbb{Q}$-factorial klt pair and let $f: Y \rightarrow X$ be a projective morphism of quasi-projective varieties. If $K_Y + \Gamma$ is $f$-big (for example, if $f$ is birational), then we may run a $(K_Y + \Gamma)$-MMP over $X$ ending with a log terminal model over $X$.
\end{theorem}

This means the MMP outputs a factorization of $f$ as a composition 
\[
Y=Y_0 \overset{g_0}{\dashrightarrow} Y_1 \overset{g_1}{\dashrightarrow} \cdots \dashrightarrow
Y_r {\rightarrow} X .
\]
If $\Gamma_{i}$ denotes the birational transform of $\Gamma$ on $Y_i$, then each pair $(Y_i, \Gamma_i)$ is $\mathbb{Q}$-factorial and klt. Each map $g_i: Y_i \dashrightarrow Y_{i+1}$ for $0 \leq i < r$ is either a divisorial contraction or a flip. Moreover, since $Y_r$ is a log terminal model, $K_{Y_r} + \Gamma_r$ is relatively nef over $X$.

The complementary case produces a Fano contraction.

\begin{theorem}[{\cite[Theorem $1.3.2$]{bchm}}]\label{bchm not pseudo-eff}
Let $(Y, \Gamma)$ be a $\mathbb{Q}$-factorial klt pair and let $f: Y \rightarrow X$ be a projective morphism of normal quasi-projective varieties. If $K_Y + \Gamma$ is not $f$-pseudo-effective, then we may run a $(K_Y+\Gamma)$-MMP over $X$ ending with a Fano contraction over $X$.
\end{theorem}

In this case, the MMP outputs a factorization of $f$ as a composition
\[
Y=Y_0 \overset{g_0}{\dashrightarrow} Y_1 \overset{g_1}{\dashrightarrow} \cdots \dashrightarrow
Y_r \overset{h}{\rightarrow} W \rightarrow X .
\]
As above, the pairs $(Y_i, \Gamma_i)$ are $\mathbb{Q}$-factorial and klt, and each map $g_i: Y_i \dashrightarrow Y_{i+1}$ for $0 \leq i < r$ is either a divisorial contraction or a flip. The final contraction $h: Y_r \rightarrow W$ is a Fano contraction. In particular, $\dim W < \dim Y_r$. This outcome will allow us to induct on dimension.

\section{RC-Invariant functors}\label{section RC-functors}

\subsection{Basic properties}\label{subsection properties}

We introduce the fundamental invariants of this paper, called RC-invariant functors. Recall that $\mathrm{Sm}$ denotes the category of smooth, connected, and non-empty quasi-projective varieties.

\begin{definition}\label{defn P^1 invt}
An \emph{RC-invariant functor} or \emph{RC-functor} is a functor $F: \mathrm{Sm} \rightarrow \mathcal{D}$ to a cocomplete $\infty$-category $\mathcal{D}$ such that, for every dominant projective morphism $f: Y \rightarrow X$ between smooth varieties whose general fiber is rationally connected, the induced morphism $F(f): F(Y) \rightarrow F(X)$ is an equivalence.
\end{definition}

This is the $\infty$-categorical version of Definition \ref{intro-defn P^1 invt} in the introduction, with $\mathcal{D}$ taken to be either the category $\mathrm{Spc}$ of spaces or the category $\mathcal{D}(R)_{\geq 0}$ of non-negatively graded chain complexes of $R$-modules. 

In particular, $F$ is birationally invariant for smooth projective varieties and $F(f)$ is an isomorphism whenever $f$ is a model for an MRC fibration. Moreover, every RC-functor satisfies the hypotheses of Theorem \ref{thm extension}. Indeed, consider a smooth blow-up square of the form \eqref{eqn abus}. Applying $F$ gives a commutative diagram
\begin{equation}\label{eqn RC}
\begin{tikzcd}
F(E) \arrow[r] \arrow[d] & F(Y) \arrow[d] \\
F(Z) \arrow[r]           & F(X)
\end{tikzcd}
\end{equation}
The morphism $Y \rightarrow X$ is birational, while $E \rightarrow Z$ is a projective-space bundle. In either case, the general fiber is rationally connected. By RC-invariance, the vertical maps in \eqref{eqn RC} are equivalences, and consequently \eqref{eqn RC} is a pushout square. 

Theorem \ref{thm extension} therefore provides an extension 
\begin{align*}
    G = RF: \Sch \rightarrow \mathcal{D}
\end{align*}
which sends abstract blow-up squares to pushout squares. By abuse of notation, we will refer to the extension $G$ as an RC-functor as well. 

We require two slight variations on RC-functors to treat the Chow group of zero-cycles (see Proposition \ref{prop CH_0}) and the structure sheaf (see Proposition \ref{prop O}), respectively.

\begin{remark}[{\bf Variants of RC-functors}]\label{rmk P^1-invt variants}
Let $\mathrm{Sm}^{\mathrm{pr}} \subset \mathrm{Sm}$ denote the subcategory whose objects are the same as $\mathrm{Sm}$, and whose maps are the proper morphisms of varieties. A \emph{properly RC-invariant functor} is a functor $F: \mathrm{Sm}^{\mathrm{pr}} \rightarrow \mathcal{D}$ such that the condition of Definition \ref{defn P^1 invt} holds.

Let $S$ be a variety, and let $\mathrm{Sm}_S$ denote the category of smooth, connected, and non-empty varieties over $S$. A \emph{relative RC-functor} is a functor $F: \mathrm{Sm}_S \rightarrow \mathcal{D}$ such that the condition of Definition \ref{defn P^1 invt} holds whenever $f$ is a morphism of varieties over $S$.
\end{remark}

The extensions of the Chow group of zero-cycles and the structure sheaf will be treated separately, since Theorem \ref{thm extension} does not directly apply. Otherwise, all subsequent results concerning RC-functors, including Theorem \ref{intro-theorem main}, apply verbatim.

We have the following construction for an RC-functor $G = RF$. In the case of the dual complex, it appears in various forms in \cite{stepanov, thu, abw, payne, dFKX}.

\begin{definition}\label{def resn complex}
Let $G = RF$ be an RC-functor, and let $Z \subset X$ be a closed subvariety. If $f: Y \rightarrow X$ is a resolution of singularities, the \emph{resolution complex} of $Z \subset X$ is defined by
\[
    G\mathcal{R}(Z \subset X) := G(f^{-1}Z).
\]
\end{definition}

To see that this construction is independent of the choice of resolution, we prove the following birational invariance statement.

\begin{lemma}\label{lemma sm blow-up compatibility}
Let $f: Y \rightarrow X$ be a projective birational morphism between smooth connected varieties, and let $Z \subset X$ be a closed subvariety. Then there is an equivalence 
\begin{align*}
    G \big( f^{-1}(Z) \big) \simeq G \big( Z \big).
\end{align*}
\end{lemma}

\begin{proof}
By the weak factorization theorem \cite{weak factorization}, it is enough to treat the case in which $f: Y \rightarrow X$ is the blow-up of $X$ along a smooth center $C \subset X$. Let $E = f^{-1}(C)$ be the exceptional divisor. Consider the abstract blow-up square
\[
\begin{tikzcd}
f^{-1}(Z\cap C) \arrow[r] \arrow[d]
& f^{-1}(Z) \arrow[d] \\
Z\cap C \arrow[r]
& Z.
\end{tikzcd}
\]
Applying $G$, we obtain a homotopy pushout square
\[
\begin{tikzcd}
G\bigl(f^{-1}(Z\cap C)\bigr) \arrow[r] \arrow[d]
& G\bigl(f^{-1}(Z)\bigr) \arrow[d] \\
G(Z\cap C) \arrow[r]
& G(Z).
\end{tikzcd}
\]

The morphism $f^{-1}(Z \cap C) \rightarrow Z \cap C$ is a projective bundle. Hence, any cubical hyperresolution of $Z \cap C$ pulls back to a cubical hyperresolution of $f^{-1}(Z \cap C)$. Since projective space is rationally connected, we deduce that $G(f^{-1}(Z \cap C)) \simeq G(Z \cap C)$. The homotopy pushout square above therefore implies that $G(f^{-1}(Z)) \simeq G(Z)$.
\end{proof}

For concrete computations of $G \mathcal{R}(Z)$, one may choose a resolution such that $f^{-1}(Z)_{\mathrm{red}}$ is an snc divisor, and then apply the standard hyperresolution of an snc divisor in Example \ref{exmp snc hyperres}. Nevertheless, the flexibility built into Definition \ref{def resn complex} will be useful. Our main Theorem \ref{intro-theorem main} may be regarded as a substantial generalization of Lemma \ref{lemma sm blow-up compatibility}.

Lemma \ref{lemma sm blow-up compatibility} also implies that an RC-invariant functor may be computed from an embedded resolution.

\begin{corollary}\label{cor X subset A}
Let $X$ be a variety. Suppose that $i: X \hookrightarrow V$ is a closed embedding into a smooth variety and $g: W \rightarrow V$ is a projective birational morphism, with $W$ smooth and $E = g^{-1}(X)_{\mathrm{red}}$ an snc divisor. Then there is a homotopy equivalence 
\begin{align*}
    G(X) \simeq G(E).
\end{align*}
\end{corollary}

This recovers the description \eqref{eqn dual resn} of the generalized dual complex in the introduction.

\subsection{Examples}\label{subsection exmps}

This section collects the principal examples of RC-invariant functors and establishes their basic properties. We defer the final example, R-equivalence classes of points, to \S \ref{section R-equ}, since it requires our main Theorem \ref{intro-theorem main}. We begin with the most basic example.

\begin{definition}\label{defn dual complex}
Let $h^0: \mathrm{Sm} \rightarrow \Spc$ be the functor sending every smooth, connected, non-empty variety to a space with one point. This is an RC-functor, and hence admits an extension $\mathcal{D} := Rh^0$. For an arbitrary variety $X$, we call $\mathcal{D}(X)$ its \emph{(generalized) dual complex}.
\end{definition}

If $X$ is smooth but not necessarily connected, then $\mathcal{D}(X)$ is the discrete space whose points correspond to the connected components of $X$ by \eqref{eqn coprod}. For singular $X$, the space $\mathcal{D}(X)$ is built from a homotopy colimit over the pieces of a hyperresolution. We compute a few examples and compare this construction with the usual dual complex.

\begin{example}[{\bf Snc curve}]
Let $E = C_1 \cup C_2$ be an snc divisor, where $C_1$ and $C_2$ are curves meeting transversely at two points $p_1$ and $p_2$. The abstract blow-up square
\[
\begin{tikzcd}
\{ p_1, p_2 \} \arrow[r] \arrow[d]
  & C_1 \arrow[d] \\
C_2 \arrow[r]
  & E
\end{tikzcd},
\]
induces a homotopy pushout square
\[
\begin{tikzcd}
S^0 \arrow[r] \arrow[d]
  & \pt \arrow[d] \\
\pt \arrow[r]
  & \mathcal{D}(E)
\end{tikzcd}.
\]
The homotopy pushout of this diagram is $\mathcal{D}(E) \simeq S^1$.
\end{example}

More generally, the generalized dual complex of an snc divisor recovers its usual dual complex.

\begin{example}[{\bf Snc divisor}]\label{exmp snc}
Let $E = \bigcup_{i =0}^n E_i$ be an snc divisor. Write $E_J = \cap_{j \in J} E_j$ for a non-empty subset $J \subset \{0, \ldots n\}$ and $E^{[k]} = \coprod_{|J| = k+1} E_J$ for $0 \leq k \leq n$. As in Example \ref{exmp snc hyperres}, the standard semi-simplicial resolution of $E$ has the form
\begin{align*}
     \bigg( \ldots E^{[2]}
    \mathrel{\substack{\rightarrow\\[-0.6ex]\rightarrow\\[-0.6ex]\rightarrow}}
    E^{[1]} \rightrightarrows E^{[0]} \bigg) \rightarrow E.
\end{align*}
By Lemma \ref{lem colimit simplicial}, $\mathcal{D}(E) \simeq \operatorname*{hocolim}_k \big( h^0 E^{[k]} \big)$. This homotopy colimit is a cell complex whose vertices correspond to the irreducible components $E_i$, whose edges correspond to connected components of pairwise intersections $E_i \cap E_j$, and more generally whose $p$-cells correspond to the connected components of the $(p+1)$-fold intersections $E_{i_0} \cap \ldots \cap E_{i_p}$. Thus $\mathcal{D}(E)$ agrees with the usual dual complex of the snc divisor $E$.
\end{example}

\begin{example}[{\bf Isolated singularities}]\label{exmp suspension}
Let $x \in X$ be an isolated singularity, and let $f: Y \rightarrow X$ be a log resolution with exceptional divisor $E = f^{-1}(x)_{\mathrm{red}}$. The associated abstract blow-up square induces a homotopy pushout
\[
\begin{tikzcd}
\mathcal{D}(E) \arrow[r] \arrow[d]
  & \mathcal{D}(Y) \simeq \pt \arrow[d] \\
\mathcal{D}(x) \simeq \pt \arrow[r]
  & \mathcal{D}(X).
\end{tikzcd},
\]
Therefore, $\mathcal{D}(X) \simeq \sum \mathcal{D}(E)$, the suspension of the usual dual complex of $E$.
\end{example}

We next introduce RC-functors measuring the low-degree cohomology groups $H^0(-, \mathbb{Z})$ and $H^1(-,\mathbb{Z})$. 

Recall that $\mathcal D(\mathbb Z)_{\geq 0}^{\mathrm{op}}$ denotes the $\infty$-category of non-negatively graded cochain complexes of abelian groups. Recall also that the integral weight piece $\gr_q^W H^{p+q}(X,\mathbb Z)$ is naturally a subquotient of $H_q^{p+q}(X,\mathbb Z)$, and that the two become isomorphic after tensoring with $\mathbb Q$.

\begin{proposition}\label{prop H^i}
The functors
\begin{align*}
H^0(-,\mathbb Z),\, H^1(-,\mathbb Z)
\colon
\mathrm{Sm}\longrightarrow \mathcal D(\mathbb Z)_{\geq 0}^{\mathrm{op}},
\end{align*}
where each cohomology group is regarded as a cochain complex concentrated in degree zero, are RC-functors. Let $RH^0,\ RH^1: \mathrm{Sch}\rightarrow \mathcal D(\mathbb Z)_{\geq 0}^{\mathrm{op}}$ denote their extensions. Then, for every projective variety $X$ and every $i\geq 0$, there are isomorphisms
\begin{align*}
    H^i \big( \mathcal{D}(X), \mathbb{Z} \big) \simeq H^i \big( RH^0(X) \big) \simeq H^i_0 \big( X , \mathbb{Z} \big)
\end{align*}
and 
\begin{align*}
    H^i \big( RH^1(X) \big) \simeq H^{i+1}_1 \big( X , \mathbb{Z} \big).
\end{align*}
After tensoring with $\mathbb Q$, these agree with the weight pieces $W_0 H^i(X, \mathbb{Q})$ and $\gr_1^W H^{i+1}(X, \mathbb{Q})$.
\end{proposition}

The restriction to degrees $0$ and $1$ is essential. For $j>1$, the groups $H^j(X,\mathbb Z)$ are not birational invariants of smooth projective varieties, and therefore cannot define RC-functors.

\begin{proof}
Let $f: Y \rightarrow X$ be a dominant projective morphism of smooth connected varieties whose general fiber is rationally connected. Clearly $H^0(X, \mathbb{Z}) \simeq H^0(Y, \mathbb{Z})$, so $H^0(-, \mathbb{Z})$ is an RC-functor. 

By Proposition \ref{prop pi1}, proved independently below, the fundamental group is RC-invariant, so $\pi_1(Y) \simeq \pi_1(X)$. The universal coefficient theorem and the Hurewicz theorem give natural identifications
\begin{align*}
    H^1(X, \mathbb{Z}) = \mathrm{Hom} \big( H_1(X, \mathbb{Z}), \mathbb{Z}\big) = \mathrm{Hom} \big( \pi_1(X), \mathbb{Z} \big)
\end{align*}
and 
\begin{align*}
    H^1(Y, \mathbb{Z}) = \mathrm{Hom} \big(H_1(Y, \mathbb{Z} \big), \mathbb{Z}) = \mathrm{Hom} \big( \pi_1(Y), \mathbb{Z} \big).
\end{align*}
Thus, $H^1(X, \mathbb{Z}) \simeq H^1(Y, \mathbb{Z})$, so $H^1(-,\mathbb{Z})$ is an RC-functor as well.

We now prove the identifications of cohomology groups. First, the functor $H^0(-,\mathbb{Z})$ factors as
\begin{align*}
    \mathrm{Sm} \xrightarrow[]{h^0} \mathrm{Spc}  \xrightarrow{C^\bullet(-,\mathbb{Z})} \mathcal{D}(\mathbb{Z})^{\mathrm{op}}_{\geq 0},
\end{align*}
where $C^\bullet(-,\mathbb{Z})$ sends a space to its singular cochain complex. Since $C^\bullet(-,\mathbb{Z})$ preserves pushout squares, the composite $C^\bullet \circ Rh^0$ sends abstract blow-up squares to pushouts. It agrees with $RH^0$ on smooth varieties, so the uniqueness part of Theorem \ref{thm extension} implies $RH^0 = C^\bullet \circ Rh^0$. Taking cohomology gives
\[
H^i\bigl(RH^0(X)\bigr)
\simeq
H^i\bigl(\mathcal D(X),\mathbb Z\bigr).
\]

On the other hand, if $X$ is projective and $\epsilon_\bullet: X_\bullet \rightarrow X$ is a cubical hyperresolution, then by definition
\begin{align*}
    RH^0(X, \mathbb{Z}) = \bigg( H^0(X_0, \mathbb{Z}) \rightarrow H^0(X_1, \mathbb{Z}) \rightarrow \ldots  \bigg).
\end{align*}
The $i$th cohomology group of this complex is $H_0^i(X, \mathbb{Z})$.

Similarly, we have an expression for the extension $RH^1(X, \mathbb{Z})$ as the total complex
\begin{align*}
    RH^1(X, \mathbb{Z}) = \bigg( H^1(X_0, \mathbb{Z}) \rightarrow H^1(X_1, \mathbb{Z}) \rightarrow \ldots  \bigg).
\end{align*}
The $i$th cohomology group of this complex is $H^{i+1}_1(X, \mathbb{Z})$. We obtain the final statement regarding the tensor product with $\mathbb{Q}$ from Proposition \ref{prop integral weight}.
\end{proof}

We next consider the fundamental group. Although $\pi_1(-)$ is itself RC-invariant, it is useful to retain slightly more homotopical information for a well-behaved extension to singular varieties. See Example \ref{exmp pi_1} for details. We therefore work with the $1$-truncation $\tau_{\leq 1} X$ of the analytic homotopy type of $X$. It is characterized up to homotopy equivalence by
\[
\pi_0(\tau_{\leq 1}X)\simeq \pi_0(X), \qquad
\pi_1(\tau_{\leq 1}X)\simeq \pi_1(X),
\]
and
\[
\pi_i(\tau_{\leq 1}X)=0
\qquad\text{for all }i>1.
\]
Equivalently, $\tau_{\leq 1} X$ is the classifying space of the fundamental groupoid $\Pi_1(X)$.

\begin{proposition}\label{prop pi1}
The functor
\begin{align*}
    \tau_{\leq 1}: \mathrm{Sm} \rightarrow \mathrm{Spc}
\end{align*}
is RC-invariant. Let $R \tau_{\leq 1}: \mathrm{Sch} \rightarrow \mathrm{Spc}$ denote its extension. Then, for every variety $X$, there is an isomorphism 
\begin{align*}
    \pi_1 \big( R \tau_{\leq 1} X \big) \simeq \pi_1(X).
\end{align*}
\end{proposition}

For a point of comparison, note that $0$-truncation defines the RC-functor  $\tau_{\leq 0} = h^0$ of Proposition \ref{prop H^i}.

\begin{proof}
Let $f: Y \rightarrow X$ be a dominant projective morphism of smooth connected varieties whose general fiber is rationally connected. By \cite[Theorem $4.1$]{bc}, there is an isomorphism $\pi_1(Y) \simeq \pi_1(X)$. Since both $X$ and $Y$ are connected, this means $\tau_{\leq 1}(Y) \simeq \tau_{\leq 1}(X)$. Hence $\tau_{\leq 1}$ is RC-invariant.

We will prove that the functor
\begin{align*}
\Pi_1\colon \mathrm{Sch}\longrightarrow \mathrm{Gpd},
\qquad
X\longmapsto \Pi_1(X),
\end{align*}
sending a variety to its fundamental groupoid carries abstract blow-up squares to pushout squares of groupoids.

Assuming this, the proposition follows. Indeed, recall that the classifying space functor induces an equivalence $B: \mathrm{Gpd} \rightarrow \mathrm{Spc}_{\leq 1}$ under which $B \Pi_1(X) = \tau_{\leq 1}(X)$. Consider the composite
\begin{align*}
    \mathrm{Sch} \xrightarrow[]{R \tau_{\leq 1}} \mathrm{Spc} \xrightarrow[]{\tau_{\leq 1}} \mathrm{Spc}_{\leq 1} \xrightarrow{B^{-1}} \mathrm{Gpd}.
\end{align*}
Since $\tau_{\leq 1}$ is left adjoint to the inclusion $\mathrm{Spc}_{\leq 1} \rightarrow \mathrm{Spc}$, it preserves pushouts. Hence, the above composition sends abstract blow-up squares to pushout squares. Moreover, on smooth varieties, it agrees with the fundamental groupoid $\Pi_1$. By the uniqueness part of Theorem \ref{thm extension}, the two functors are equivalent. In particular, taking fundamental groups, we obtain $\pi_1( R \tau_{\leq 1} X ) \simeq \pi_1(X)$.

Consider now an abstract blow-up square of the form \eqref{eqn abus}. Choose compatible triangulations of the analytic spaces underlying $X$ and $Y$ such that $Z$ and $E$ are subcomplexes, respectively. Let $U \subset X$ be the open star of $Z$, and set $V = X \setminus Z$. Then $U$ is a small open neighborhood of $Z$ which deformation retracts onto $Z$. By the Seifert-Van Kampen theorem for fundamental groupoids, 
\begin{align}\label{eqn pi X}
    \Pi_1 \big(X \big) \simeq \Pi_1 \big(V \big) *_{\Pi_1 ( U \cap V ) } \Pi_1 \big( U \big).
\end{align}
The pre-image $f^{-1}(U)$ is the open star of $E$, so it deformation retracts onto $E$. By Seifert-Van Kampen again,
\begin{align}\label{eqn pi Y}
    \Pi_1 \big( Y \big) \simeq \Pi_1 \big( f^{-1}(V) \big) *_{\Pi_1 ( f^{-1}(V \cap U) )} \Pi_1 \big(f^{-1}(U) \big).
\end{align}
Using \eqref{eqn pi X} and \eqref{eqn pi Y}, together with the identifications $f^{-1}(V) \simeq V$ and $f^{-1}(V \cap U) \simeq V \cap U$, we compute
\begin{align*}
    \Pi_1 \big(Y \big) *_{\Pi_1(E)} \Pi_1 \big( Z \big) \simeq& \bigg( \Pi_1 \big(V \big) *_{\Pi_1(V \cap U)} \Pi_1 \big(f^{-1}(U) \big) \bigg) *_{\Pi_1(f^{-1}(U))} \Pi_1 \big(U \big) \\
    \simeq& \Pi_1 \big(V \big) *_{\Pi_1(V \cap U)} \Pi_1 \big(U \big) \\ \simeq& \Pi_1 \big(X \big).
\end{align*}
Thus \(\Pi_1\) sends abstract blow-up squares to pushout squares of groupoids.
\end{proof}

\begin{remark}[{\bf Simpler variants}]
Alternatively, one may consider the groupoid-valued functor
\begin{align*}
    \Pi_1: \mathrm{Sm} \rightarrow \mathrm{Gpd}, \qquad X \mapsto \Pi_1(X).
\end{align*}
The argument of Proposition~\ref{prop pi1} shows it is RC-invariant and that its extension is given by $R\Pi_1(X) = \Pi_1(X) \in \mathrm{Gpd}$. This is conceptually simpler than the functor in Proposition \ref{prop pi1}, whose geometric meaning in higher degrees is unclear. Nevertheless, the latter functor retains strictly more information, so we shall continue to use it. 

A similar comment will apply to the Chow group of zero-cycles: there is an RC-functor $CH_0: \mathrm{Sm} \rightarrow \mathrm{AbGp}$ whose extension satisfies $RCH_0(X) \simeq CH_0(X) \in \mathrm{AbGp}$ for all projective varieties. Instead, we will work with an analogue that retains more information.
\end{remark}

It is crucial in Proposition \ref{prop pi1} to record the full $1$-truncation $\tau_{\leq 1} X$ (or the fundamental groupoid), rather than only the fundamental group $\pi_1(X)$. The following example illustrates why.

\begin{remark}[{\bf Fundamental Group}]\label{rmk groupoid}
Let $\mathbb{P}^1 \rightarrow C$ be the normalization of a nodal rational curve. Applying $R \tau_{\leq 1}$ to the associated abstract blow-up square gives a pushout square
\[
\begin{tikzcd}
S^0 \arrow[r] \arrow[d]
  & \pt \arrow[d] \\
\pt \arrow[r]
  & R \tau_{\leq 1} C.
\end{tikzcd}.
\]
Its homotopy pushout is $R \tau_{\leq 1}(C) \simeq S^1$, and therefore $\pi_1(C) \simeq \pi_1(S^1) \simeq \mathbb{Z}$. The generator arises from identifying the two points of the fiber over the node. If one only records fundamental groups, this information is lost, and the resulting pushout fails to recover $\pi_1(C)$.
\end{remark}

We next consider the structure sheaf and its cohomology as examples of RC-functors. We use the notion of a relative RC-functor introduced in Remark \ref{rmk P^1-invt variants}. For a variety $S$, let $\mathcal{D}_{\mathrm{qc}}(\mathscr{O}_S)_{\geq 0}$ denote the $\infty$-category of non-negatively graded complexes of quasi-coherent $\mathscr{O}_S$-modules.

\begin{proposition}\label{prop O}
Let $S$ be a variety. The assignment
\begin{align*}
    \mathscr{O}_{/S}: \mathrm{Sm}_S \rightarrow \mathcal{D}_{\mathrm{qc}}(\mathscr{O}_S)_{\geq 0}, \hspace{.4cm} \big( \phi:X \rightarrow S \big) \mapsto R \phi_* \mathscr{O}_X
\end{align*}
is a relative RC-invariant functor. It admits an extension $R \mathscr{O}_{/S}: \mathrm{Sch}_S \rightarrow \mathcal{D}_{\mathrm{qc}}(\mathscr{O}_S)_{\geq 0}$ given by $R\mathscr{O}_{/S} (X) \simeq R \phi_* \underline{\Omega}_X^0$, which sends abstract blow-up squares to pushout squares.
\end{proposition}

\begin{proof}
For RC-invariance, it suffices to show that if $f: Y \rightarrow X$ is a dominant projective morphism of smooth connected varieties with rationally connected general fibers, then $R f_* \mathscr{O}_Y \simeq \mathscr{O}_X$. This is established in \cite[Theorem $7.1$]{dualizing} when $X$ and $Y$ are projective. The general case follows by choosing projectivizations and then restricting.

By Mayer-Vietoris \cite[Proposition $4.11$]{du bois}, the Du Bois complexes of an abstract blow-up square form a pushout square. Consequently, $R\phi_* \underline{\Omega}_X^0$ provides an extension.
\end{proof}

\begin{proposition}\label{prop structure}
Fix $j \geq 0$. The functor
\begin{align*}
    H^j(-, \mathcal{O}): \mathrm{Sm} \rightarrow \mathcal{D}(\mathbb{C})^{op}_{\geq 0}, \qquad X \mapsto H^j(X, \mathscr{O}_X)
\end{align*}
is RC-invariant. Let $RH^j(-,\mathscr{O}): \mathrm{Sch} \rightarrow \mathcal{D}(\mathbb{C})^{op}_{\geq 0}$ denote its extension. Then, for every variety $X$ and integer $i \geq 0$,
\begin{align*}
    H^i( RH^j(X, \mathcal{O})) \simeq \gr_0^F \gr_j^W H^{i+j}(X, \mathbb{C}).
\end{align*}
\end{proposition}

\begin{proof}
RC-invariance follows from Proposition \ref{prop O} by taking cohomology groups. Let $X$ be a variety with cubical hyperresolution $\epsilon_\bullet: X_\bullet \rightarrow X$. By construction,
\begin{align*}
    RH^j(X, \mathscr{O}) \simeq \bigg( H^j(X_0, \mathscr{O}_{X_0}) \rightarrow H^j(X_1, \mathscr{O}_{X_1}) \rightarrow \ldots \bigg).
\end{align*}
This complex is obtained by applying $\gr_0^F$ to the $j$th row of the weight spectral sequence \eqref{eqn ratl weight ss}. Since the Hodge filtration is strict, taking cohomology gives
\begin{align*}
    H^i\big( RH^j(X, \mathscr{O}) \big) \simeq \gr_0^F \gr_j^W H^{i+j} \big( X, \mathbb{C} \big),
\end{align*}
as claimed.
\end{proof}

We now consider the Chow group of zero-cycles. This requires some care, as $CH_0$ is functorial only for proper morphisms. Accordingly, let $\mathrm{Sm}^{\mathrm{pr}} \subset \mathrm{Sm}$ and $\mathrm{Sch}^{\mathrm{pr}} \subset \mathrm{Sch}$ denote the subcategories with the same objects as $\mathrm{Sm}$ and $\mathrm{Sch}$, respectively, and with maps the proper morphisms.

\begin{proposition}\label{prop CH_0}
The functor
\begin{align*}
    CH_0: \mathrm{Sm}^{\mathrm{pr}} \rightarrow \mathcal{D}(\mathbb{Z})_{\geq 0}, \hspace{.4cm} X \mapsto CH_0(X),
\end{align*}
is properly RC-invariant functor. It admits an extension $RCH_0: \mathrm{Sch}^{\mathrm{pr}} \rightarrow \mathcal{D}(\mathbb{Z})_{\geq 0}$ which sends abstract blow-up squares to pushout squares. Moreover, for every projective variety $X$,
\begin{align*}
    H_0 \big( RCH_0(X)) \simeq CH_0(X \big).
\end{align*}
\end{proposition}

\begin{proof}
By \cite[Theorem $2.5$]{chow constant}, the functor $CH_0$ is properly
RC-invariant. To construct its extension, we apply the argument of \cite[Theorem $2.2$]{guillen navarro}. This approach is more flexible than the one developed in \cite{infinity cat}; see \cite[Variant $2.5$]{guillen navarro} and subsequent examples for further applications. It is important that we work in the abelian category of abelian groups to access the formalism of descent categories. The resulting extension can be computed using hyperresolutions, as in Theorem \ref{thm extension}.

For the final isomorphism, it is enough to show that $CH_0$ sends every abstract blow-up square to a pushout of abelian groups, as in the proof of Proposition \ref{prop pi1}.

Consider an abstract blow-up square of the form \eqref{eqn abus}. Since $f \amalg i: Y \amalg Z \rightarrow X$ is dominant, the induced map
\begin{align*}
    CH_0(Y) \oplus_{CH_0(E)} CH_0(Z) \rightarrow CH_0(X)
\end{align*}
is surjective. Moreover, the classes of any two points in $Y \amalg Z$ with the same image in $X$ are identified in the pushout $ CH_0(Y) \oplus_{CH_0(E)} CH_0(Z)$. To prove injectivity, it therefore suffices to show that every relation $[ \mathrm{div} (\phi)] = 0$ in $CH_0(X)$, defined by a rational function $\phi \in \mathbb{C}(C)^\times$ on an irreducible curve $C \subset X$, lifts to the pushout. If $C \subset Z$, the relation is accounted for in $CH_0(Z)$. Otherwise, the strict transform $\widetilde{C} \subset Y$ maps birationally onto $C$, and so the relation is accounted for in $CH_0(Y)$.
\end{proof}

Our final example concerns codimension two cycles. For a smooth projective variety $X$, define 
\begin{align*}
    F^2 CH^2(X) := \ker \bigg( AJ_X:CH^2(X)_{\hom} \rightarrow J^3(X) \bigg),
\end{align*}
where $CH^2(X)_{\hom} \subset CH^2(X)$ denotes the subgroup of null-homologous codimension two cycles. The notation is inspired by the Bloch-Beilinson conjecture, which posits that this is one piece of a filtration on Chow groups. 

It is straightforward to see that $F^2 CH^2(X)$ is a birational invariant of smooth projective varieties. Moreover, if $CH_0(X) \simeq \mathbb{Z}$, then $F^2 CH^2(X) = 0$ by \cite[Theorem $6.24$]{voisin lectures}.

This group arises in the study of the negative $K$-group $K_{-n+2}(X)$ in \cite[Theorem $9.6$]{bps}. In that work we used an ad hoc version of the MMP argument developed here. We now show that $F^2 CH^2(X)_\mathbb{Q}$ fits naturally into this framework as an RC-invariant functor.

\begin{proposition}\label{prop F^2}
The assignment
\begin{align*}
    F^2CH^2_\mathbb{Q}: \mathrm{Sm} \rightarrow \mathcal{D}(\mathbb{Q})_{\geq 0}^{\mathrm{op}}, \qquad X \mapsto F^2 CH^2(\overline{X})_\mathbb{Q},
\end{align*}
where $\overline{X}$ is a smooth projectivization of $X$, defines an RC-invariant functor. 
\end{proposition}

For a smooth variety $X$, the group $CH^2(\overline{X})_\mathbb{Q}$ is independent of the choice of projectivization, since it is a birational invariant. Likewise, a morphism $f: Y \rightarrow X$ induces, after choosing compatible projectivizations, a well-defined map $CH^2(\overline{X})_\mathbb{Q} \rightarrow CH^2(\overline{Y})_\mathbb{Q}$. Thus the assignment is functorial.

\begin{proof}
As in the case of $CH_0$, birational invariance shows that $F^2 CH^2(\overline{X})_\mathbb{Q}$ is independent of the choice of projectivization and is functorial. 

Let $f: Y \rightarrow X$ be a dominant morphism of smooth projective varieties whose general fiber is rationally connected. Since $f$ is dominant, the pullback $CH^2(X)_\mathbb{Q} \rightarrow CH^2(Y)_\mathbb{Q}$ is injective. Hence so is its restriction
\begin{align*}
    f^*: F^2 CH^2(X)_\mathbb{Q} \rightarrow F^2 CH^2(Y)_\mathbb{Q}.
\end{align*}
Moreover, the general fiber $F$ is rationally connected, so $CH_0(F) \simeq \mathbb{Z}$. By the preceding discussion, $F^2 CH^2(F) = 0$. Therefore the hypotheses of \cite[Theorem $0.5$]{saito} are satisfied. The theorem implies that $f^*$ is surjective and hence an isomorphism.
\end{proof}

\subsection{Proof of the rigidity theorem}\label{subsection pf main thm}

The goal of this section is to prove our main Theorem \ref{intro-theorem main}. Throughout, let $G = RF$ be the extension of an RC-functor $F$. The argument applies without modification when $F$ is a properly RC-functor or a relative RC-functor, as in Remark \ref{rmk P^1-invt variants}.

\begin{definition}
A dominant projective morphism $f: Y \rightarrow X$ is \emph{$G$-compatible} if for any closed subvariety $Z \subset X$, the map $G(f^{-1}Z) \rightarrow G(Z)$ is an equivalence.

A variety $X$ is \emph{$G$-trivial} if it admits a $G$-compatible resolution $f:Y \rightarrow X$.
\end{definition}

By Lemma \ref{lemma sm blow-up compatibility}, a variety is $G$-trivial if and only if every resolution is $G$-compatible. The following lemma shows that it is enough to verify $G$-compatibility on irreducible subvarieties.

\begin{lemma}\label{lem reduce irred}
A morphism $f: Y \rightarrow X$ is $G$-compatible if and only if for any irreducible subvariety $Z \subset X$, the map $G(f^{-1}(Z))\rightarrow G(Z)$ is an equivalence.
\end{lemma}

\begin{proof}
The only if direction is obvious. Conversely, suppose that $G(f^{-1}(Z)) \simeq G(Z)$ is an equivalence for any irreducible subvariety. We argue the same holds for all subvarieties $Z \subset X$ by induction on $\dim Z$ and on the number of components of $Z$. Write $Z = Z_1 \cup Z_2$, where $Z_1$ is irreducible and $Z_2$ has strictly fewer irreducible components than $Z$.  Consider the two pushout squares
\[
\begin{tikzcd}[row sep=2.2em, column sep=2.6em]
G \big( Z_1 \cap Z_2 \big) \arrow[r] \arrow[d] & G \big( Z_1 \big) \arrow[d] \\
G \big( Z_2 \big) \arrow[r] & G \big( Z \big),
\end{tikzcd}
\qquad
\begin{tikzcd}[row sep=2.2em, column sep=2.6em]
G \big( f^{-1}(Z_1 \cap Z_2) \big) \arrow[r] \arrow[d] & G \big( f^{-1}(Z_1) \big) \arrow[d] \\
G \big( f^{-1}(Z_2) \big) \arrow[r] & G \big( f^{-1}(Z) \big).
\end{tikzcd}
\]
By the induction hypothesis, the natural maps give equivalences $G( f^{-1}(Z_1) ) \simeq G(Z_1)$, $G( f^{-1}(Z_2)) \simeq G(Z_2)$, and $G( f^{-1}(Z_1 \cap Z_2)) \simeq G(Z_1 \cap Z_2)$. It follows that the diagrams have equivalent pushouts $G(f^{-1}(Z)) \simeq G(Z)$.
\end{proof}

We may now state our main theorem. 

\begin{theorem}[{\(\!=\) Theorem~\ref{intro-theorem main}}]\label{theorem main}
Suppose \(f:Y\to X\) is a dominant projective morphism between connected varieties of klt type whose general fiber is rationally connected. Then \(f\) is \(G\)-compatible.
\end{theorem}

The proof is by induction on $n = \dim Y$. The base case $n = 0$ is immediate. We maintain this notation throughout, assuming by induction that Theorem \ref{intro-theorem main} is known in dimension $< n$. There are two main cases to consider. The birational case, where $n = \dim X$, is treated in Proposition \ref{prop main A}. The case of positive-dimensional fibers, where $n > \dim X$, is treated in Proposition \ref{prop main B}.

\begin{remark}[{\bf Slight refinement}]\label{rmk R-equ}
The proof actually gives a slightly stronger statement, which we will need in \S \ref{section R-equ} to establish that the set of R-equivalence classes is an an RC-functor. Let $G: \mathrm{Sch} \rightarrow \mathcal{D}$ be a functor which sends abstract blow-up squares to pushout squares and which satisfies Definition \ref{defn P^1 invt} for all varieties of dimension $< n = \dim Y$, but possibly not in higher dimensions. Then, in the setting of Theorem \ref{theorem main}, one has $G(f^{-1}(Z)) \simeq G(Z)$ for every proper subvariety $Z \subsetneq X$.
\end{remark}

We will prove Theorem \ref{intro-theorem main} by running a suitable minimal model program and analyzing the resulting contractions. Let us start with the behavior of $G$ with respect to divisorial and flipping contractions.

\begin{lemma}\label{lem div cont}
Assume Theorem \ref{intro-theorem main} in dimension $< n$. Let $(Y, \Gamma)$ be a $\mathbb{Q}$-factorial dlt pair and suppose $f: Y \rightarrow X$ is a divisorial contraction corresponding to a $(K_Y + \Gamma)$-negative extremal ray. If $f$ contracts a component $E_i$ of $\Gamma^{=1}$, then $f$ is $G$-compatible.
\end{lemma}

\begin{proof}
Consider the restriction $g = f|_{E_i}: E_i \rightarrow T$. We claim that $T$ is of klt type. Granted this, since $\dim E_i < n$, our induction hypothesis implies $g$ is $G$-compatible.

By the vanishing theorem in \cite[Theorem $8.1$]{fujino log mmp} applied to the lc center $E_i$, we have $(f|_{E_i})_*\mathcal{O}_{E_i} \simeq \mathcal{O}_{T}$. Since $E_i$ is normal, this implies $T$ is normal. To show $T$ is of klt type, we apply the canonical bundle formula. Perturb the boundary on the dlt pair $(E_i, \Gamma_{E_i})$ slightly so that it is klt and $-(K_{E_i}+\Gamma_{E_i})$ is still $g$-ample. Take a general divisor $H \in | -m(K_{E_i}+\Gamma_{E_i})|_{T}$ for $m \gg 0$ so that $(E_i, \Gamma_{E_i} + \frac{1}{m} H)$ is still klt. By construction,
\begin{align*}
    K_{E_i} + \Gamma_{E_i} + \frac{1}{m} H \sim_{\mathbb{Q},g} 0.
\end{align*}
Since $T$ is normal, the canonical bundle formula in \cite[Theorem $7.1$]{canonical bundle} implies that $T$ is of klt type.

Now let $Z \subset X$ be a closed subvariety. Consider the pushout square
\begin{equation*}
\begin{tikzcd}
G \big( f^{-1}(Z \cap T)) \big) \arrow[r] \arrow[d] & G \big( f^{-1}(Z) \big) \arrow[d] \\
G\big( Z \cap T \big) \arrow[r] & G \big( Z \big).
\end{tikzcd}
\end{equation*}
Since $g = f|_{E_i}$ is $G$-compatible, we have an equivalence $G(f^{-1}(Z \cap T_i)) \simeq G(Z \cap T_i)$. The pushout square shows $G(f^{-1}(Z)) \simeq G(Z)$.
\end{proof}

\begin{lemma}\label{lem flip}
Assume Theorem \ref{intro-theorem main} in dimension $< n$. Let $(Y, \Gamma)$ be a $\mathbb{Q}$-factorial dlt pair, and suppose $f: Y \dashrightarrow X$ is a flip corresponding to a $(K_Y + \Gamma)$-negative extremal ray, with flipping contractions $h: Y \rightarrow V$ and $h^+: X \rightarrow V$. If there exists a component $E_i$ of $E = \Gamma^{=1}$ containing the flipping locus $\mathrm{Ex}(h)$ and a component $E_j^+$ of $E^+ = \Gamma_X^{=1}$ containing the flipped locus $\mathrm{Ex}(h^+)$, then $h$ and $h^+$ are $G$-compatible.
\end{lemma}

\begin{proof}
Consider the flipping diagram
\[
\begin{tikzcd}[column sep=4em, row sep=3em]
E_i \subset Y \arrow[dr, "h"'] \arrow[rr, dashed,  "f"] & & X \supset E_j^+ \arrow[dl, "h^{+}"] \\
& V &
\end{tikzcd}
\]
The pair $(V, \Gamma_V)$ is dlt (but not $\mathbb{Q}$-factorial). Hence, the components $h(E_i)$ and $h^+(E_j^+)$ of $\Gamma_W^{=1}$ are of klt type. Our inductive hypothesis implies that the restrictions $h|_{E_i}: E_i \rightarrow h(E_i)$ and $h^+|_{E_j^+}: E_j^+ \rightarrow h^+(E_j^+)$ are $G$-compatible.

Let $Z \subset V$ be a closed subvariety and set $W = Z \cap h(\mathrm{Ex}(h)) = Z \cap h(\mathrm{Ex}(h^+))$. Thus, $W$ is contained in the images $h(E_i)$ and $h^+(E_j^+)$. Consider the two pushout squares
\begin{equation}\label{eqn flip}
\begin{tikzcd}
G \big( h^{-1}(W) \big) \arrow[r] \arrow[d] & G \big(h^{-1}(Z) \big) \arrow[d] \\
G \big( W \big) \arrow[r] & G \big( Z \big),
\end{tikzcd}
\qquad
\begin{tikzcd}
G \big( (h^{+})^{-1} (W) \big) \arrow[r] \arrow[d] & G \big( (h^{+})^{-1} (Z) \big) \arrow[d] \\
G \big( W \big) \arrow[r] & G \big( Z \big).
\end{tikzcd}
\end{equation}
Since $h|_{E_i}$ and $h^+|_{E_j^+}$ are $G$-compatible, there are equivalences $G(h^{-1}(W)) \simeq G(W)$ and $G( h^{+})^{-1} (W)) \simeq G(W)$. The pushout squares \eqref{eqn flip} therefore show $G(h^{-1}(Z)) \simeq G(Z)$ and $G ( (h^{+})^{-1} (Z)) \simeq G(Z)$.
\end{proof}

We are now prepared to prove Theorem \ref{intro-theorem main} in the case where $n = \dim X$. By taking a resolution of $Y$, we immediately reduce to the case where $Y$ is smooth. Therefore, it suffices to prove the following.

\begin{proposition}\label{prop main A}
Assume Theorem \ref{intro-theorem main} in dimension $< n$. If $f: Y \rightarrow X$ is a resolution of a variety $X$ of klt type, then $f$ is $G$-compatible.
\end{proposition}

\begin{proof}
Suppose first that $X$ is $\mathbb{Q}$-factorial. Choose a boundary $\mathbb{Q}$-divisor $\Delta$ on $X$ so that the pair $(X, \Delta)$ is klt. Blowing up $Y$ and applying Lemma \ref{lemma sm blow-up compatibility}, we can arrange that $f$ is a log resolution with exceptional divisor $E$. Consider the boundary $\mathbb{Q}$-divisor $\Gamma = E + f_*^{-1} \Delta$ on $Y$. We claim that we can run an MMP on the dlt pair $( Y, \Gamma )$ over $X$, just as in the klt case appearing in Theorem \ref{bchm pseudo-eff}. 

For this, consider the auxiliary $\mathbb{Q}$-divisor
\begin{equation}\label{eqn auxiliary}
\begin{split}
A &= E + f_*^{-1}\Delta - \Delta_Y \\
  &\sim_{\mathbb{Q}} K_Y + \Gamma - f^*(K_X+\Delta).
\end{split}
\end{equation}
Here, we are using the log pull-back $\Delta_Y$ of $\Delta$ defined by \eqref{eqn delta_Y}. By definition, since $(X, \Delta)$ has klt singularities, the coefficients of $\Delta_Y$ are all $< 1$. Hence, $A$ is an effective divisor, with support consisting precisely of the exceptional components $E_i$ of $E$. 

Set $\Gamma_\epsilon = \Gamma - \epsilon A$ for $0 < \epsilon \ll 1$. It is effective and the pair $(Y, \Gamma_\epsilon)$ is klt. Moreover, by the second line of \eqref{eqn auxiliary},
\begin{align*}
    K_Y + \Gamma_\epsilon = K_Y + \Gamma - \epsilon A \sim_{\mathbb{Q}, X} (1- \epsilon) \big(K_Y + \Gamma \big).
\end{align*}
Thus $(K_Y + \Gamma)$ and $(K_Y + \Gamma_\epsilon)$ determine the same negative extremal rays over $X$. Applying Theorem \ref{bchm pseudo-eff}, we obtain an MMP for $(Y, \Gamma_\epsilon)$ over $X$, and this is the same as an MMP on $(Y, \Gamma)$ over $X$.

The MMP outputs a factorization of $f$ as a composition 
\[
Y=Y_0 \overset{g_0}{\dashrightarrow} Y_1 \overset{g_1}{\dashrightarrow} \cdots \dashrightarrow
Y_r {\rightarrow} X .
\]
The pairs $(Y_i, \Gamma_i)$ are dlt, where $\Gamma_i$ is the pushforward of $\Gamma$ to $Y_i$. Denote by $f_i: Y_i \rightarrow X$ the morphism to $X$. 

We claim that $f_r: Y_r \rightarrow X$ is an isomorphism. Indeed, the second line of \eqref{eqn auxiliary} shows the pushforward $A_r$ of $A$ to $Y_r$ is $\mathbb{Q}$-linearly equivalent to the divisor $K_{Y_r} + \Gamma_r$ over $X$. Since $Y_r$ is a log terminal model, this is relatively nef over $X$. But then $A_r$ is effective, exceptional, and relatively nef, so it must be trivial by the negativity lemma \cite[Lemma $3.39$]{km}. In other words, $g_r$ contracts all components of $E$. Since $X$ is $\mathbb{Q}$-factorial, the small morphism $f_r: Y_r \rightarrow X$ must be an isomorphism.

It therefore suffices to analyze each $g_i$ separately. If $g_i: Y_i \rightarrow Y_{i+1}$ is a divisorial contraction, it contracts a component of $\Gamma_i$. By Lemma \ref{lem div cont}, $g_i$ is $G$-compatible. If $g_i = (h_i^+)^{-1} \circ h_i$ is a flip, then $\mathrm{Ex}(h_i) \subset \mathrm{Ex}(f_i)$. Indeed, since $X$ is $\mathbb{Q}$-factorial, $\mathrm{Ex}(f_i)$ is divisorial and so it is equal to the support of $\Gamma_i$. Similarly, $\mathrm{Ex}(h_i^+) \subset \mathrm{Ex}(f_{i+1}) = \mathrm{Supp}(\Gamma_{i+1})$. Lemma \ref{lem flip} implies that $h_i$ and $h_i^+$ are $G$-compatible. Putting this all together, we conclude that $f$ is $G$-compatible.

Let us now prove that an arbitrary klt type variety $X$ is $G$-compatible. Let $Z \subset X$ be a closed subvariety, which we may assume is irreducible by Lemma \ref{lem reduce irred}. By Lemma \ref{lem plt blow-up}, there exists a $\mathbb{Q}$-factorial dlt pair $(Y, E)$, a projective birational morphism $f: Y \rightarrow X$, an open neighborhood $U \subset X$ containing the generic point $\eta_Z$, and an irreducible component $E_0$ of $E$ such that $f^{-1}(Z \cap U) = E_0 \cap U$. It suffices to prove that $G(f^{-1}(Z)) \rightarrow G(Z)$ is an equivalence. Indeed, we have already seen that $Y$ is $G$-compatible, so $G(f^{-1}(Z)) \simeq G( ( f \circ r_Y)^{-1}(Z))$ for a resolution $r_Y: \widetilde{Y} \rightarrow Y$. 

Let $W = Z \cap (X \setminus U)$, and take a resolution $a: \widetilde{Z} \rightarrow Z$ of $Z$. To show $G(f^{-1}(Z)) \simeq G(Z)$, our goal is to identify the two pushout squares

\begin{equation*}
\begin{tikzcd}
G \big( f^{-1}(W) \cap E_0 \big) \arrow[r] \arrow[d] & G \big( E_0 \big) \arrow[d] \\
G \big( f^{-1}(W) \big) \arrow[r] & G \big( f^{-1}(Z) \big),
\end{tikzcd}
\qquad
\begin{tikzcd}
G \big( a^{-1}(W) \big) \arrow[r] \arrow[d] & G \big( \widetilde{Z} \big) \arrow[d] \\
G \big( W \big) \arrow[r] & G \big( Z \big).
\end{tikzcd}
\end{equation*}

By induction on $\dim Z$, we know $G(W) \simeq G( (f \circ r_Y)^{-1}(W))$, and the latter is equivalent to $G(f^{-1}(W))$ because $Y$ is $\mathbb{Q}$-factorial of klt type.

Choose a resolution $\widetilde{E}_0$ of $E_0$ sitting in a commutative diagram
\begin{equation*}
\begin{tikzcd}
\widetilde{E}_0 \arrow[r, "b"] \arrow[d, "c"] & \widetilde{Z} \arrow[d, "a"] \\
E_0 \arrow[r, "f|_{E_0}"] & Z.
\end{tikzcd}    
\end{equation*}
Since $E_0$ is klt type of dimension $< n$, our inductive hypothesis implies that $c$ is $G$-compatible. Moreover, every fiber of $f$ and hence also of $f|_{E_0}$ is rationally chain connected by Theorem \ref{thm hm}$(1)$. The general fiber of $b$ is a resolution of the general fiber of $f|_{E_0}$, the latter of which is of klt type. By Theorem \ref{thm hm}$(1)$ again, the general fiber of $b$ is rationally connected and so $b$ is $G$-compatible by our inductive hypothesis.

Using the $G$-compatibility of $b$ and $c$, we obtain equivalences $G(\widetilde{Z}) \simeq G(\widetilde{E}_0) \simeq G(E_0)$ and
\begin{align*}
 G \big( a^{-1}(W) \big) \simeq G \big( b^{-1} a^{-1}(W) \big) \simeq G \big( c^{-1} f|_{E_0}^{-1} (W) \big) \simeq G \big( f|_{E_0}^{-1}(W) \big)
\end{align*}
But $f|_{E_0}^{-1}(W) = f^{-1}(W) \cap E_0$. We conclude that the diagrams have equivalent pushouts $G(f^{-1}(Z)) \simeq G(Z)$.
\end{proof}

We will now use Proposition \ref{prop main A} to help prove Theorem \ref{intro-theorem main} in the case $\dim X < n = \dim Y$. First, we reduce to the case where $X$ and $Y$ are smooth. For this, choose resolutions $r_X: \widetilde{X} \rightarrow X$ and $r_Y: \widetilde{Y} \rightarrow Y$ fitting in a commutative diagram
\[
\begin{tikzcd}
\widetilde{Y} \arrow[r, "\widetilde{f}"] \arrow[d,"r_Y"] & \widetilde{X} \arrow[d, "r_X"] \\
Y \arrow[r, "f"] & X.
\end{tikzcd}
\]
By Proposition \ref{prop main A}, both $r_Y$ and $r_X$ are $G$-compatible. Moreover, the general fiber of $\widetilde{f}$ is rationally connected. Indeed, it is a resolution of a general fiber of $f$, which is rationally chain connected of klt type, so we may apply Theorem \ref{thm hm}$(1)$. A diagram chase therefore shows that it suffices to prove $\widetilde{f}$ is $G$-compatible.

Secondly, we must understand the behavior of RC-functors with respect to Fano contractions.

\begin{lemma}\label{lem fano fib}
Assume Theorem \ref{intro-theorem main} in dimension $< n$. Let $(Y, \Gamma)$ be a $\mathbb{Q}$-factorial dlt pair and let $h: Y \rightarrow W$ be a Fano contraction corresponding to a $(K_Y + \Gamma)$-negative extremal ray. If $E = \Gamma^{=1}$ is $h$-vertical and $D = h(E)$, then $G(E) \simeq G(D)$.
\end{lemma}

\begin{proof}
Write $E = \bigcup_{i \in I} E_i$ and $E_J = \bigcap_{j \in J} E_j$ for each nonempty subset $J \subset I$. As in the case where $E$ is an snc divisor in Example \ref{exmp snc hyperres}, we have $G(E) \simeq \operatorname*{colim} G(E_J)$.

By Lemma \ref{lemma qdlt fano}, there is a $\mathbb{Q}$-divisor $\Gamma_W$ on $W$ such that $(W, \Gamma_W)$ is qdlt and $D = \Gamma_W^{=1}$. Moreover, the strata $D_J$ of $D$ correspond bijectively with the strata $E_J$ of $E$. Under this correspondence, the restriction maps $E_J \rightarrow D_J$ are between varieties of klt with rationally connected general fibers. By the inductive hypothesis, $G(E_J) \simeq G(D_J)$ for each nonempty subset $J \subset I$. 

Applying the same colimit description to $D$, we have $G(D) \simeq \operatorname*{colim} G(D_J)$. The equivalences on the strata are compatible with the transition maps and therefore induce an equivalence $G(E) \simeq G(D)$.
\end{proof}

We are now prepared to prove the following proposition. Together with Proposition \ref{prop main A}, this will complete the proof of Theorem~\ref{intro-theorem main}.

\begin{proposition}\label{prop main B}
Assume Theorem \ref{intro-theorem main} in dimension $< n$. If $f: Y \rightarrow X$ is a dominant projective morphism between smooth connected varieties whose general fiber is rationally connected and positive-dimensional, then $f$ is $G$-compatible.
\end{proposition}

\begin{proof}
Fix a closed subvariety $Z \subset X$. If $Z = X$, then $G(Y) \simeq G(X)$ by the definition of an RC-functor. Otherwise, by Lemma \ref{lem reduce irred}, we may assume that $Z \subsetneq X$ is a proper irreducible subvariety. After blowing up $X$ and $Y$ and applying Lemma \ref{lemma sm blow-up compatibility}, we reduce to the case where $Z$ is purely divisorial and $E = f^{-1}(Z)_{\mathrm{red}}$ is an snc divisor. 

Run an MMP on the pair $(Y, E)$ over $X$. Since $Z$ is Cartier, this is equivalent to running an MMP on the klt pair $(Y, E - \epsilon f^*(Z))$ for $0 < \epsilon \ll 1$. By Theorem \ref{bchm not pseudo-eff}, the latter MMP can be run and it terminates with a Fano contraction. Indeed, $K_Y+E $ is not $f$-pseudo-effective, since the general fiber of $f$ is rationally connected of positive dimension. 

We obtain a factorization of $f$ of the form
\[
Y=Y_0 \overset{g}{\dashrightarrow} 
Y_r \overset{h}{\rightarrow} W \overset{k}{\rightarrow} X .
\]
Here, $g$ is the composition of the flips and divisorial contractions. The pair $(Y_r, E_r)$ is dlt, where  $E_r = g_* E$. Moreover, since $Z$ is Cartier, the pre-image $(k \circ h)^{-1}(Z)$ is divisorial and so it is equal to $E_r$. By Proposition \ref{prop main A}, we have $G(E) \simeq G(E_r)$.

The morphism $h: Y_r \rightarrow W$ is a Fano contraction. Hence, Lemma \ref{lem fano fib} implies $G(E_r) \simeq G(h(E_r))$. Moreover, $W$ is of klt type of dimension $< n$. Since $k: W \rightarrow X$ has rationally connected general fibers, the inductive hypothesis gives $G(k^{-1}(Z)) \simeq G(Z)$. Because $k^{-1}(Z) = h(E_r)$, these equivalences combine to yield $G(Z) \simeq G(E)$, as required.
\end{proof}

\section{Applications}

\subsection{Rational connectedness of strata}\label{section collapse}

Theorem \ref{intro-theorem main} gives strong control over global invariants of the fibers of a rationally connected fibration. In many applications, however, one needs finer information about the individual components and strata of a fiber. This section develops such a refinement.

As a motivating example, consider the following question of Koll\'ar \cite[Question $16$]{ax}.

\begin{question}[{= Question \ref{intro-question}}]
Let $f: Y \rightarrow C$ be a dominant projective morphism from a smooth connected variety to a smooth curve whose general fiber is rationally connected. If $x \in C$ and $E = f^{-1}(x)_{\mathrm{red}}$ is an snc divisor, must $E$ have a rationally connected irreducible component?
\end{question}

In this setting, Theorem \ref{intro-theorem main} gives
\[
G(E)\simeq G(\mathrm{pt})
\]
for every RC-functor \(G\). In particular, the dual complex $\mathcal{D}(E)$ is contractible. Moreover, $E$ is rationally chain connected, for example by \cite[Theorem IV.\(3.5\)]{rational curves}. However, these methods do not by themselves distinguish between the contributions of individual strata. 

Brown and Foster answered Kollár's question affirmatively in \cite[Remark $5.7$]{brown foster}. We complete the picture by describing the rational connectedness behavior of all strata of $E$. The main technical result is as follows.

\begin{theorem}\label{thm cartier collapse}
Let $f: Y \rightarrow X$ be a dominant projective morphism whose general fiber is rationally connected. Assume that $(X, D)$ is plt with $D$ a $\mathbb{Q}$-Cartier divisor, and that $(Y, \Gamma)$ is $\mathbb{Q}$-factorial qdlt with $E = \Gamma^{=1}$ equal to $f^{-1}(D)_{\mathrm{red}}$. Then $E$ has an irreducible component $E_0$ such that the induced morphism $E_0 \rightarrow D$ has rationally connected general fibers. 

Moreover, the strata of $E$ other than $E_0$ admit an ordered partition into pairs
\[
    \big( Z_1, W_1 \big), \ldots , \big( Z_a, W_a \big)
\]
such that, for every $i$,
\begin{enumerate}
    \item $Z_i$ is a divisor in $W_i$ with $f(Z_i) = f(W_i)$;
    \item $Z_i$ and $W_i$ have birational MRC bases; and 
    \item every stratum of $E$ contained in $W_i$ is equal to $Z_j$ or $W_j$ for some $j \leq i$.
\end{enumerate}
\end{theorem}

To make the first two conditions explicit, consider the Stein factorization of $f|_{W_i}$:
\begin{align*}
    W_i \xrightarrow[]{g_i} U_i \rightarrow f(W_i).
\end{align*}
We prove that $g_i$ has rationally connected generals fibers and that the restriction $g_i|_{Z_i}: Z_i \rightarrow U_i$ is dominant with rationally connected general fibers. Hence, $f(Z_i) = f(W_i)$, and $Z_i$ and $W_i$ have birational MRC bases.

As discussed in the introduction, the ordered partition yields a sequence of elementary collapses of  the dual complex $\mathcal{D}(E)$. Each pair $(Z_i, W_i)$ determines one such collapse, and the corresponding pairs of cells can be removed successively until only the vertex corresponding to $E_0$ remains. Thus, $\mathcal{D}(E)$ is collapsible. This is a combinatorial property stronger than contractibility.

The collapsibility of dual complexes is studied further in \cite{dFKX}, and in fact their arguments inspire the proof of Theorem \ref{thm cartier collapse}. The key additional feature here is that we track rational connectedness throughout the collapse: each pair has birational MRC bases.

In preparation for the proof, let us establish some terminology. Suppose $(Y, \Gamma)$ is a qdlt pair (see Definition \ref{defn qdlt}) and $f: Y \dashrightarrow X$ is a divisorial contraction or flip corresponding to a  $(K_Y+\Gamma)$-negative extremal ray. The \emph{exceptional locus} $\mathrm{Ex}(f)$ is the locus in $Y$ on which $f$ is not an isomorphism. If $S$ is an lc center of $(Y, \Gamma)$, we say $f$ \emph{contracts} $S$ if $S \subset \mathrm{Ex}(f)$. By \cite[Lemma $3.38$]{km}, the lc centers of the qdlt pair $(X, f_*\Gamma)$ are precisely the strict transforms of the non-contracted lc centers of $(Y, \Gamma)$.

The following lemma tracks the behavior of contracted lc centers under divisorial contractions and flips.

\begin{lemma}\label{lem collapse}
Let $(Y, \Gamma)$ be a $\mathbb{Q}$-factorial qdlt pair and suppose $f: Y \dashrightarrow X$ is a divisorial contraction or flip corresponding to a $(K_Y+\Gamma)$-negative extremal ray $R$. Suppose there exists an irreducible component $E_0$ of $E = \Gamma^{=1}$ such that $(E_0 \cdot R) > 0$. 

Then the contracted strata of $E$ admit an ordered partition into pairs $(Z_1, W_1), \ldots, (Z_a, W_a)$ satisfying conditions \textup{(1)}--\textup{(3)} of Theorem \ref{thm cartier collapse}.
\end{lemma}

\begin{proof}
In this situation, the proof of \cite[Theorem $19$]{dFKX} establishes a bijection
\begin{equation}\tag*{}\label{eq:lc-centers-bijection}
\begin{gathered}
\left\{
\begin{array}{c}
\text{contracted lc centers } W \text{ of } (Y,\Gamma) \\
\text{such that } W \not\subset E_0
\end{array}
\right\}
\;\simeq\;
\left\{
\begin{array}{c}
\text{contracted lc centers } Z \text{ of } (Y,\Gamma) \\
\text{such that } Z \subset E_0
\end{array}
\right\} \\
W \mapsto Z = W \cap E_0.
\end{gathered}
\end{equation}

Consider a pair $(Z, W)$ arising from this bijection, and let $h: Y \rightarrow V$ denote the contraction morphism. Applying \cite[Proposition $25$]{dFKX} to the generic point of $h(W)$, we find that, since $E_0$ meets every fiber of $h$ and $W \not \subset E_0$, the intersection $W \cap E_0$ has a unique irreducible component dominating $h(W)$. This component is $Z$. In particular, $h(W) = h(Z)$, so condition $(1)$ is satisfied.

To verify condition $(2)$, consider now the Stein factorization $W \xrightarrow[]{g} U \rightarrow h(W)$ of $h|_W$. Since $-(K_Y+\Gamma)$ is $h$-ample, the restriction $-(K_W + \Gamma_W)$ is $h|_W$-ample and hence also $g$-ample. After a small perturbation of the boundary, Theorem \ref{thm hm}$(2)$ shows that $g$ has rationally chain connected fibers. On the other hand, the restriction $g|_Z: Z \rightarrow U$ has connected fibers by the Shokurov-Kollár connectedness lemma. Moreover, $-(K_Z+ \Gamma_Z)$ is $g|_Z$-ample, so these fibers are rationally chain connected, again by Theorem \ref{thm hm}$(2)$. This implies that $Z$, $U$, and $W$ all have birational MRC bases, so condition $(2)$ is satisfied.

Finally, choose an ordering of the contracted pairs by ascending dimension $\dim W_i$. We show condition $(3)$ is satisfied. Indeed, if $W_j \subsetneq W_i$, then $\dim W_j < \dim W_i$, so $j < i$ by construction. If $Z_j \subset W_i$, then either $j < i$ or else $\dim Z_j = \dim Z_i$. But $Z_j \subset D_0$, so then $Z_j \subset D_0 \cap W_i = Z_i$. Since they have the same dimension, $Z_j = Z_i$, hence $j = i$. This concludes the proof.
\end{proof}

We are now ready to prove Theorem~\ref{thm cartier collapse}. The strategy is to apply Lemma~\ref{lem collapse} at each step of a suitable MMP. The
assumption that \(D\) is a \(\mathbb Q\)-Cartier divisor ensures that its preimage remains divisorial throughout the process. Without this assumption, exceptional components may appear in codimension at least two, where the collapse argument is harder to control.

\begin{proof}[Proof of Theorem \ref{thm cartier collapse}]
We induct on the relative dimension $\dim Y - \dim X$.

First suppose that $f$ is birational. Run an MMP on the qdlt pair $(Y, \Gamma)$ over $X$. Since $D$ is $\mathbb{Q}$-Cartier, this is the same as running an MMP on the klt pair $(Y, \Gamma - \epsilon f^*(D))$ over $X$ for $0 < \epsilon \ll 1$. By Theorem \ref{bchm pseudo-eff}, the MMP terminates and gives a factorization of $f$ as
\[
Y=Y_0 \overset{g_0}{\dashrightarrow} Y_1 \overset{g_1}{\dashrightarrow} \cdots \dashrightarrow
Y_r {\rightarrow} X .
\]

Each pair $(Y_i, \Gamma_i)$ is qdlt, where $\Gamma_i$ is the pushforward of $E$. Moreover, since $D$ is $\mathbb{Q}$-Cartier, its preimage in $Y_i$ is divisorial, and hence is equal to $\Gamma_i^{=1} = E_i$. For $0 \leq i < r$, the map $g_i: Y_i \dashrightarrow Y_{i+1}$ is either a divisorial contraction or a flip. By \cite[Lemma $21$]{dFKX}, either $g_i$ does not contract any lc centers, or the hypotheses of Lemma \ref{lem collapse} are satisfied. Applying Lemma \ref{lem collapse} successively, we pair off the strata contracted at each step of the MMP. We order these pairs first by the step $i$ at which they are contracted, and then according to the ordering provided by Lemma \ref{lem collapse}. This ordering satisfies properties \((1)\)--\((3)\). Indeed, if a stratum \(S\) is contracted, then every stratum contained in \(S\) is contracted at the same or an earlier step, which gives property \((3)\).

We claim that $(Y_r, \Gamma_r)$ has a unique lc center, namely the strict transform of $D$. Indeed, as in the proof of Proposition \ref{prop main A}, consider the auxiliary $\mathbb{Q}$-divisor
\begin{align*}
    A &= \Gamma - \Delta_Y \\ &\sim_{\mathbb{Q}} K_Y + \Gamma - f^*(K_X+D).
\end{align*}
Since $(X, D)$ is plt, the coefficients of $\Delta_Y$ are $< 1$, except for the strict transform of $D$. Hence $A$ is effective, and its support is precisely the union of the $f$-exceptional divisors. On the other hand, the pushforward $g_{r*}A$ is $\mathbb{Q}$-linearly equivalent to $K_{Y_r} + \Gamma_r$ over $X$. Since $Y_r$ is a log terminal model, this is relatively nef over $X$. But then $g_{r*}A$ is effective, exceptional, and relatively nef, so it must be trivial by the negativity lemma \cite[Lemma $3.39$]{km}. In other words, $g_r$ contracts all $f$-exceptional divisors. Taking $E_0$ to be the strict transform of $D$, this completes the proof in the birational case.

Now suppose $\dim Y > \dim X$. Run an MMP on the pair $(Y, \Gamma)$ over $X$. This is the same as running an MMP on $(Y, \Gamma - \epsilon f^*(D))$ for $0 < \epsilon \ll 1$, so it terminates with a Fano contraction by Theorem \ref{bchm not pseudo-eff}. The hypothesis that $K_Y+\Gamma$ not be $f$-pseudo-effective is satisfied, because the general fiber of $f$ is rationally connected of positive dimension. 

We obtain a factorization of $f$ of the form
\[
Y=Y_0 \overset{g}{\dashrightarrow} 
Y_r \overset{h}{\rightarrow} W \rightarrow X,
\]
where $g$ is a composition of flips and divisorial contractions and $h$ is a Fano contraction. As in the birational case, \cite[Lemma $21$]{dFKX} allows us to apply Lemma \ref{lem collapse} at each step of $g$. This pairs the strata contracted during the MMP; we order these pairs first by the step at which they are contracted and then by the ordering provided by Lemma \ref{lem collapse}. 

By Lemma \ref{lemma qdlt fano}, there exists a $\mathbb{Q}$-divisor $\Gamma_W$ such that $(W, \Gamma_W)$ is $\mathbb{Q}$-factorial qdlt, and $h$ induces a bijection between the strata of $(Y_r, \Gamma_r)$ and those of $(W, \Gamma_W)$. Moreover, the map $S \rightarrow h(S)$ between corresponding strata has rationally connected general fibers. We now apply the induction hypothesis to the morphism \(W\to X\). Adjoining the resulting pairs of strata after those obtained from the preceding MMP steps gives the required ordered partition. This completes the induction and proves the theorem.
\end{proof}

We have the following immediate consequence in the setting of Question \ref{intro-question}.

\begin{corollary}[{\(\!=\) Theorem~\ref{intro-thm degeneration}}]\label{cor degeneration}
Let \(f:Y\to C\) be a dominant projective morphism from a smooth connected variety to a smooth curve whose general fibers are rationally connected. Suppose $x \in C$ is a closed point and $E = f^{-1}(x)_{\mathrm{red}}$ is an snc divisor. Then \(E\) has a rationally connected irreducible component \(E_0\).

Moreover, the strata of $E$ other than $E_0$ admit an ordered partition into pairs 
\begin{align*}
    \big( Z_1,W_1 \big),\ldots, \big( Z_a,W_a \big) 
\end{align*}
satisfying conditions \textup{(1)}--\textup{(3)} of Theorem \ref{thm cartier collapse}.
\end{corollary}

It is also useful to apply Theorem \ref{thm cartier collapse} to resolutions of klt singularities. As mentioned in the introduction, an ad hoc version of the following argument plays a key role in the study of the algebraic $K$-group $K_{-n+2}(X)$, where $n = \dim X$. Indeed, \cite[Theorem $12.1$]{bps} shows that if $X$ has isolated singularities of klt type, then $K_{-n+2}(X)_\mathbb{Q} \simeq W_2 H^n(X, \mathbb{Q})$.

\begin{corollary}[{= Theorem \ref{intro-thm collapse}}]\label{corollary collapse}
Let $(X, \Delta)$ be a klt pair and $Z \subset X$ an irreducible subvariety. Then there exists a resolution $f: Y \rightarrow X$ such that $E = f^{-1}(Z)_{\mathrm{red}}$ is an snc divisor and the following holds. There exists an irreducible component $E_0$ of $E$ such that the map $E_0 \rightarrow Z$ is dominant with rationally connected general fibers, and the remaining strata of $E$ that dominate $Z$ admit an ordered partition into pairs 
\begin{align*}
    \big( Z_1,W_1 \big),\ldots, \big( Z_a,W_a \big) 
\end{align*} 
satisfying conditions \textup{(1)}--\textup{(3)} of Theorem \ref{thm cartier collapse}.

If $X$ is $\mathbb{Q}$-factorial, these properties hold for every resolution $f: Y \rightarrow X$ such that $f^{-1}(Z)_{\mathrm{red}}$ is an snc divisor.
\end{corollary}

\begin{proof}
Apply Lemma \ref{lem plt blow-up}. After shrinking $X$ around the generic point $\eta_Z$ of $Z$, this gives a $\mathbb{Q}$-factorial plt pair $(V, D)$ and a projective birational morphism $g: V \rightarrow X$ such that $D = g^{-1}(Z)$. Choose a resolution $h: Y \rightarrow V$ such that $h^{-1}(D)_{\mathrm{red}}$ is an snc divisor, and set $f = g \circ h$. We apply Theorem \ref{thm cartier collapse} to the morphism $h$. Discarding all pairs of strata that do not dominate $Z$, this proves the first statement.

For the second statement, assume that $X$ is $\mathbb{Q}$-factorial, and let $f: Y \rightarrow X$ be a resolution such that $f^{-1}(Z)_{\mathrm{red}}$ is an snc divisor. Run an MMP on $(Y, f^{-1}(Z)_{\mathrm{red}} + f_*^{-1} \Delta)$ over $X$. This is possible by the same perturbation argument as in the proof of Proposition \ref{prop main A}. We now follow the proof of Theorem \ref{thm cartier collapse}. At each divisorial contraction or flip, the hypotheses of Lemma \ref{lem collapse} are satisfied by \cite[Lemma $23$]{dFKX} (which requires that $X$ be $\mathbb{Q}$-factorial). Discarding all pairs of contracted strata that do not dominate $Z$ and taking $E_0$ to be the final lc center dominating $Z$, this completes the proof.
\end{proof}

When $X$ is not $\mathbb{Q}$-factorial, it is unknown whether the conclusion of Corollary \ref{corollary collapse} holds for every such resolution. The method of proof can break down, as shown by studying the quadric cone in \cite[Example $17$]{dFKX}. The issue is that a morphism $Y_i \rightarrow X$ arising in the MMP may have exceptional components of codimension $\geq 2$, in which case Lemma \ref{lem collapse} need not apply.

\subsection{R-equivalence classes}\label{section R-equ}

In this section, we show that the set $X(\mathbb{C})/R$ of $R$-equivalence classes of complex points defines an RC-functor. The argument is more involved than it was for the other examples in \S \ref{subsection exmps}, relying on the refinement of Theorem \ref{intro-theorem main} discussed in Remark \ref{rmk R-equ}.

\begin{definition}\label{defn R-equ}
Let $X$ be a complex variety. Two closed points $x,y\in X(\mathbb{C})$ are \emph{R-equivalent} in the sense of \cite{manin} if they can be joined by a chain of rational curves in $X$. We denote the resulting set of equivalence classes by $X(\mathbb{C})/R$.
\end{definition}

We now prove the following proposition.

\begin{proposition}\label{prop R-equ}
The functor
\begin{align*}
    \mathrm{Sm} \rightarrow \mathrm{Spc}, \qquad X \mapsto X(\mathbb{C})/R
\end{align*}
is RC-invariant. Let $R(\mathbb{C}): \mathrm{Sch} \rightarrow \mathrm{Spc}$ denotes its extension. Then, for every variety $X$, there is an isomorphism
\begin{align*}
    \pi_0 \big( R(\mathbb{C})(X) \big) \simeq X(\mathbb{C})/R.
\end{align*}
\end{proposition}

This is a fundamental example of an RC-functor because it detects rational chain connectedness. Namely, the extension space $R(\mathbb{C})(X)$ is connected if and only if $X$ is rationally chain connected.

\begin{proof}
We first show that the assignment $X \mapsto X(\mathbb{C})/R$ sends abstract blow-up squares to pushout squares of sets. Given an abstract blow-up square \eqref{eqn abus}, consider the induced map
\begin{align*}
    \big( Y(\mathbb{C})/R \big) \amalg_{E(\mathbb{C})/R} \big( Z(\mathbb{C})/R  \big)\rightarrow X(\mathbb{C})/R
\end{align*}
It is surjective because $f \amalg i: Y \amalg Z \rightarrow X$ is dominant. Moreover, the classes of any two points in $Y \amalg Z$ with the same image in $X$ are identified in the pushout. To prove injectivity, it therefore suffices to show that every $R$-equivalence relation generated by a rational curve $C \subset X$ lifts to the pushout. If $C \subset Z$, the relation is accounted for in $Z(\mathbb{C})/R$. Otherwise, the strict transform $\widetilde{C} \subset Y$ maps birationally onto $C$, so the relation is accounted for in $Y(\mathbb{C})/R$.

In particular, restricting to smooth blow-up squares, the functor $X \mapsto X(\mathbb{C})/R$ satisfies the hypotheses of Theorem \ref{thm extension}. It therefore admits an extension $R(\mathbb{C}): \mathrm{Sch} \rightarrow \mathrm{Spc}$.

Let us show that, for every variety $X$, there is an isomorphism $\pi_0 \big( R(\mathbb{C})(X) \big) \simeq X(\mathbb{C})/R$. As in the proofs of Propositions \ref{prop pi1} and \ref{prop CH_0}, consider the composite
\begin{align*}
    \mathrm{Sch} \xrightarrow[]{R(\mathbb{C})} \mathrm{Spc} \xrightarrow[]{\pi_0} \mathrm{Set}.
\end{align*}
Since $\pi_0$ preserves colimits, the composite sends abstract blow-up squares to pushouts. On smooth varieties, it agrees with the assignment \(X\mapsto X(\mathbb{C})/R\), regarded as a functor to \(\mathrm{Set}\). We have seen that the latter also sends abstract blow-up squares to pushouts. Therefore, by the uniqueness part of Theorem \ref{thm extension}, the two functors are naturally isomorphic. 

Finally, we prove RC-invariance. Let $f: Y \rightarrow X$ be a dominant projective morphism between smooth connected varieties whose general fiber is rationally connected. We argue by induction on $n = \dim Y$. Since $f$ is dominant, the induced map
\begin{align*}
    Y(\mathbb{C})/R \rightarrow X(\mathbb{C})/R
\end{align*}
is surjective. To prove injectivity, it suffices to show that for every point or rational curve $Z \subset X$, the preimage $f^{-1}(Z)$ is rationally chain connected. 

For every proper subvariety $Z \subsetneq X$, Remark \ref{rmk R-equ} and the inductive hypothesis imply that
\[
   R(\mathbb{C}) \big( f^{-1}(Z) \big) \simeq R(\mathbb{C}) \big(Z \big).
\]
By considering the associated sets of connected components, it follows that $f^{-1}(Z)$ is rationally chain connected whenever $Z \subsetneq X$ is a  point or rational curve.

It remains to consider the cases in which \(X\) itself is a point or a rational curve. If $X$ is a point, then $Y$ is rationally connected by assumption. If $X$ is a rational curve, then $Y$ is rationally connected by \cite[Corollary $1.3$]{ghs}.
\end{proof}

\begin{remark}
The same method yields an alternative proof that the Chow group of zero-cycles \(CH_0\) defines an RC-functor, since \(CH_0\) is governed by the behavior of points and curves in the same way as \(R\)-equivalence.
\end{remark}

\subsection{Dual complex, weights, and further applications}\label{subsection applications}

In this section, we collect applications of our main Theorem \ref{intro-theorem main} to various RC-functors. We begin with the dual complex.

\begin{corollary}[{= Theorem \ref{intro-thm dual complex}}]\label{cor dual complex}
Suppose $f: Y \rightarrow X$ is a projective morphism between connected varieties of klt type whose general fiber is rationally connected. Then for any closed subvariety $Z \subset X$, there is a homotopy equivalence
\begin{align*}
    \mathcal{D} \big( f^{-1}(Z) \big) \simeq \mathcal{D} \big( Z \big).
\end{align*} 
\end{corollary}

\begin{proof}
Apply Theorem \ref{intro-theorem main} to the RC-functor $\mathcal{D} = R h^0$ of Definition \ref{defn dual complex}.
\end{proof}

The preceding corollary describes the homotopy invariance of the dual complex under rationally connected fibrations. We next relate its rational cohomology to the weight-zero part of the cohomology of $X$ and apply vanishing results for $W_0 H^i(X, \mathbb{Q})$ proved in the appendix.

\begin{corollary}[{\(\Rightarrow\) Proposition \ref{intro-prop weight zero}}]
Let $X$ be a projective variety. For every $i \geq 0$, there is an isomorphism
\begin{align*}
    H^i(\mathcal{D}(X), \mathbb{Q}) \simeq W_0 H^i(X, \mathbb{Q}).
\end{align*}
Consequently:
\begin{enumerate}
    \item If $X$ has rational singularities, then $\mathcal{D}(X)$ is $\mathbb{Q}$-contractible.
    \item If $X$ is equidimensional and $s = \dim \mathrm{Sing}(X)$, then $H^i( \mathcal{D}(X), \mathbb{Q}) = 0$ for all $0 < i < \mathrm{depth} \mathscr{O}_X - s$. In particular, if $X$ has isolated Cohen-Macaulay singularities, then $\mathcal{D}(X)$ is a $\mathbb{Q}$-homology wedge of spheres.
\end{enumerate}
\end{corollary}

\begin{proof}
By Proposition \ref{prop H^i}, we have $H^i(\mathcal{D}(X), \mathbb{Q}) \simeq W_0 H^i(X, \mathbb{Q})$. Then $(1)$ and $(2)$ follow from Lemma \ref{lem weight restrictions} and Corollary \ref{cor O depth}, respectively.
\end{proof}

\begin{remark}[{\bf Dual complex for rational singularities}]\label{rmk stepanov}
Suppose that \(x\in X\) is an isolated rational singularity, and let \(E\) be the exceptional divisor of a log resolution. Stepanov \cite{stepanov} asked whether the dual complex \(\mathcal{D}(E)\) must be contractible. By Proposition \ref{intro-prop weight zero}\((1)\), together with Example \ref{exmp suspension}, the dual complex \(\mathcal{D}(E)\) is \(\mathbb Q\)-contractible. However, it need not be contractible: \cite[Example 8.1]{payne} constructs a counterexample with \(\mathcal{D}(E) \simeq \mathbb{RP}^2\). More generally, \cite[Theorem $3$]{kollar prescribed} shows that every \(\mathbb Q\)-contractible finite cell complex occurs as the dual complex of an exceptional divisor over an isolated rational singularity.
\end{remark}

We now turn to integral coefficients. Recall that the graded pieces $\gr_q^W H^{p+q}(X, \mathbb{Z})$ are subquotients of the terms $E_2^{p,q} = H^{p+q}_q(X, \mathbb{Z})$ in the integral weight spectral sequence \eqref{eqn weight ss}, and they coincide after tensoring with $\mathbb{Q}$. The bottom two rows $q = 0, 1$ define RC-invariant functors. 

\begin{corollary}[{\(\Rightarrow\) Corollary \ref{intro-cor integral weight}}]\label{cor klt sing}
Suppose $f: Y \rightarrow X$ is a morphism between connected projective varieties of klt type whose general fiber is rationally connected. Then for any closed subvariety $Z \subset X$ and every $i \geq 0$, there are isomorphisms 
\begin{align*}
    H^i_0 \big( Z, \mathbb{Z} \big) \simeq H^i_0 \big( f^{-1}(Z), \mathbb{Z} \big) \hspace{.3cm} \text{and} \hspace{.3cm} H^i_1 \big(Z, \mathbb{Z} \big) \simeq H^i_1 \big(f^{-1}(Z), \mathbb{Z} \big).
\end{align*}
Thus, if $X$ is projective and of klt type, then
\begin{align*}
    H^i(X, \mathbb{Z}) \text{ has weights } \geq \min \{ 2, i\}.
\end{align*}
\end{corollary}

\begin{proof}
Apply Theorem \ref{intro-theorem main} to the RC-functors $RH^0$ and $RH^1$ of Proposition \ref{prop H^i}. For the latter statement, take $Z = X$ and $f$ a resolution of singularities. We obtain $H^i_0(X, \mathbb{Z}) = H^i_1(X, \mathbb{Z}) = 0$ for all $i > 0$. Since each group $\gr_q^W H^{p+q}(X, \mathbb{Z})$ is a subquotient of $H_q^{p+q}(X, \mathbb{Z})$, the desired vanishing follows.
\end{proof}

By way of comparison, Lemma \ref{lem weight restrictions} says that if \(X\) is projective with rational singularities, then
\[
    H^i(X,\mathbb Q)\text{ has weights }\ge \min\{2,i\}.
\]
Since varieties of klt type have rational singularities, Corollary \ref{cor klt sing} gives an integral refinement of this rational weight vanishing.
We suspect that no such refinement holds for general rational singularities, but we are unable to prove this. Indeed, Remark \ref{rmk stepanov} gives examples for which $E_2^{i,0} = H^i_0(X, \mathbb{Z}) \neq 0$, but it remains unclear which, if any, of these classes survive to the $E_\infty$-page of the weight spectral sequence. 

We now turn to the more geometric RC-functors introduced in \S \ref{subsection exmps}.

\begin{corollary}[{\(\Rightarrow\) Corollary \ref{intro-cor fibers}}]\label{cor preimages}
Let $f: Y \rightarrow X$ be a dominant projective morphism between connected varieties of klt type whose general fiber is rationally connected. Then for any closed subvariety $Z \subset X$, the following hold:
\begin{enumerate}
    \item There is a bijection
        \begin{align*}
            \big( f^{-1}(Z) \big) (\mathbb{C})/R \simeq Z(\mathbb{C}) / R.
        \end{align*}
    \item There is an isomorphism
        \begin{align*}
            RCH_0 \big( f^{-1}(Z) \big) \simeq RCH_0 \big( Z \big).
        \end{align*}
        In particular, 
        \begin{align*}
            CH_0 \big( f^{-1}(Z) \big) \simeq CH_0 \big( Z \big).
        \end{align*}
    \item There is a quasi-isomorphism
        \begin{align*}
            R f_* \underline{\Omega}_{f^{-1}(Z)}^0 \simeq \underline{\Omega}_Z^0.
        \end{align*}
    \item There is an isomorphism
        \begin{align*}
            \pi_1 \big( f^{-1}(Z) \big) \simeq \pi_1 \big( Z \big).
        \end{align*}
\end{enumerate}
\end{corollary}

\begin{proof}
We apply Theorem \ref{intro-theorem main} to the RC functors $R(\mathbb{C})$ of Proposition \ref{prop R-equ}, $RCH_0$ of Proposition \ref{prop CH_0}, $R \mathscr{O}_{/X}$ of Proposition \ref{prop structure}, and $R \tau_{\leq 1}$ of Proposition \ref{prop pi1}.
\end{proof}

Taking $Z$ to be a point recovers Corollary \ref{intro-cor fibers}. At the opposite extreme, let $f: Y \rightarrow X$ be a resolution and take $Z = X$. Corollary \ref{cor preimages} implies that if $X$ is of klt type, then there is a quasi-isomorphism 
\begin{align*}
    Rf_* \mathscr{O}_Y \simeq \underline{\Omega}_X^0. 
\end{align*}
In the terminology of \cite{parkpopa}, this says that $X$ satisfies condition $D_0$. The statement itself is not new: varieties of klt type are well-known to have rational singularities, and a variety has rational singularities if and only it is Du Bois and satisfies condition $D_0$. Nevertheless, it is noteworthy that condition \(D_0\) emerges directly from RC-functors.

In forthcoming work \cite{burke CH_0}, we use Hodge-theoretic methods to prove that statement $(3)$ remains valid under the weaker assumptions that $X$ and $Y$ satisfy condition $D_0$ and the general fiber $F$ is connected with $H^i(F, \mathscr{O}_F) = 0$ for all $i > 0$. Likewise, more refined Chow-theoretic arguments prove statement $(2)$, assuming only that the general fiber satisfies $CH_0(F) \simeq \mathbb{Z}$. The following example shows that no analogous result holds for the fundamental group.

\begin{example}[{\bf Simply-connected families}]\label{exmp pi_1}
Let $S$ be a K3 surface equipped with a fixed-point-free involution $\iota$, and let $E$ be an elliptic curve with involution $x \mapsto -x$. The diagonal action of $\mathbb{Z}/2\mathbb{Z}$ on $S \times E$ is fixed-point-free, and it induces a morphism
\[
    f\colon
    Y:= (S\times E)/(\mathbb Z/2\mathbb Z)
    \longrightarrow
    E/(\mathbb Z/2\mathbb Z)\simeq \mathbb P^1.
\]

The general fiber of $f$ is isomorphic to $S$, and hence is simply-connected. Over each of the four branch points, however, the fiber is the Enriques surface $S / \iota$. Hence,
\begin{align*}
    \pi_1( S/\iota) \simeq \mathbb{Z} / 2 \mathbb{Z}.
\end{align*}
Moreover, $\pi_1(Y) \not\simeq \pi_1( \mathbb{P}^1 ) = 1$.
\end{example}

This paper studies RC-functors over $\mathbb{C}$. We conclude by recording an analogous phenomenon for point-counting over finite fields \cite{esnault, fr, witt rational}. We do not know whether it reflects any formal connection with the results above.

\begin{remark}[{\bf Point counts}]
Let $\mathbb{F}_q$ be a finite field. If $Y / \mathbb{F}_q$ is a smooth projective variety which is rationally chain connected over $\overline{k(Y)}$, then \cite{esnault} proves
\begin{align*}
    \# Y(\mathbb{F}_q)  \equiv 1 \pmod q.
\end{align*}
More generally, it suffices to assume that $CH_0(Y_{\overline{k(Y)}}) \simeq \mathbb{Z}$. An extension to varieties with Witt-rational singularities is established in \cite{witt rational}.

There is a relative version that is reminiscent of the situation for RC-functors. Indeed, \cite{fr} proves that if $f: Y \rightarrow X$ is a dominant projective morphism between smooth connected varieties over $\mathbb{F}_q$ whose generic fiber $F$ is rationally chain connected over $\overline{k(Y)}$, then
\[
    \#f^{-1}(x)(\mathbb{F}_q)\equiv 1 \pmod q
    \qquad\text{for every }x\in X(k).
\]

In particular,
\[
    \#f^{-1}(Z)(\mathbb{F}_q)\equiv \# Z(\mathbb{F}_q) \pmod q
    \qquad\text{for every }Z \subset X.
\]
More generally, the same conclusion holds under the weaker assumption $CH_0(F_{\overline{k(Y)}})_{\mathbb{Q}} \simeq \mathbb{Q}$.
\end{remark}

\section{Appendix}

\subsection{Weak Lefschetz for singular varieties}

In this appendix, we prove a Lefschetz hyperplane theorem for singular varieties, Theorem~\ref{thm lefschetz}, formulated in terms of the depth of Du Bois complexes. Applying this theorem to the weight filtration yields Corollary \ref{cor depth lefschetz}, and this is used in the main text to prove Proposition \ref{intro-prop weight zero}. We recall only the Hodge-theoretic background essential to these arguments, referring the interested reader to \cite{parkpopa} for a more detailed treatment.

Theorem~\ref{thm lefschetz} is motivated by the Lefschetz hyperplane theorem of \cite[Theorem B]{parkpopa}, which is formulated in terms of the \emph{local cohomological defect} $\lcdef(X)$. Suppose that $X$ is embedded in a smooth variety $Y$. Writing $\mathcal{H}^q_X(\mathcal{O}_Y)$ for the local cohomology sheaves, the first nonzero such sheaf occurs in degree $q = \codim_Y X$. One defines
\begin{align*}
    \lcdef(X) := \max \{ q : \mathcal{H}^q_X(\mathcal{O}_Y) \neq 0 \} - \codim_Y X.
\end{align*}
This invariant is independent of the chosen embedding. If $X$ has quotient singularities or local complete intersection singularities, then $\lcdef(X) = 0$; see \cite[Example $2.4$]{parkpopa} for further examples and discussion.

\begin{theorem}\cite[Theorem B]{parkpopa}\label{thm pp B}
Let $X$ be an equidimensional projective variety of dimension $n$, $D$ an ample effective Cartier divisor on $X$, and $U = X \setminus D$. Then the restriction map
\[
    H^i(X, \mathbb{Q}) \rightarrow H^i(D, \mathbb{Q})
\]
is an isomorphism for $i \leq n-2 - \lcdef(U)$ and injective for $i = n-1-\lcdef(U)$.
\end{theorem}

The local cohomological defect admits a description in terms of finer holomorphic invariants encoded by Du Bois complexes. Let 
\[
\underline{\Omega}_X^p \in
D^b_{\mathrm{coh}}(\mathscr{O}_X)
\]
denote the \emph{$p$th Du Bois complex} of $X$. It serves as a Hodge-theoretic replacement for the sheaf of holomorphic $p$-forms on a smooth variety. Moreover, it can be computed in terms of a hyperresolution. If $n= \dim X$, then by \cite[Theorem E]{local cohomology},
\begin{align*}
    \lcdef(X) = n - \min_{0 \leq p \leq n} \{ \depth \underline{\Omega}_X^p + p \}.
\end{align*}

This formula suggests refining the Lefschetz theorem by keeping track of the individual depths \(\depth \underline{\Omega}_X^p\), rather than only
their minimum. The resulting refinement gives sharper control over the restriction map in specified Hodge degrees.

\begin{theorem}\label{thm lefschetz}
Let $X$ be an equidimensional projective variety of dimension $n$, $D$ be an ample effective Cartier divisor on $X$, and $U = X \setminus D$. If $\depth \underline{\Omega}_U^p+p \geq k$ for all $p \leq m$, then the restriction map 
\begin{align*}
    H^{p,q}(X) \rightarrow H^{p,q}(D)
\end{align*}
is an isomorphism for $p+q \leq k-2$ and $\min \{ p,q \} \leq m$. Moreover, it is injective for $p + q \leq k-1$ and $\min \{ p, q-1 \} \leq m$.
\end{theorem}

The theorem recovers the Lefschetz Theorem \ref{thm pp B} as a special case. Indeed, if we take $k = n - \lcdef(U)$, the formula for $\lcdef(U)$ above shows that
\[
    \depth \underline{\Omega}_U^p+p\geq k
    \qquad\text{for all } p \geq 0.
\]

Theorem~\ref{thm lefschetz} also implies the Lefschetz Theorem
\cite[Theorem~9.2]{parkpopa}, which is formulated in terms of condition
$D_m$ or duality up to level $m$.

\begin{corollary}\cite[Theorem $9.2$]{parkpopa}
Let $X$ be an equidimensional projective variety of dimension $n$, and let $D$ be an ample effective Cartier divisor on $X$. Suppose that $U = X \setminus D$ satisfies condition $D_m$. Then the restriction map 
\begin{align*}
    H^{p,q}(X) \rightarrow H^{p,q}(D)
\end{align*}
is an isomorphism for $p+q \leq n-2$ and $\min \{ p,q \} \leq m$. Moreover, it is injective for $p + q \leq n-1$ and $\min \{ p, q-1 \} \leq m$.
\end{corollary}

\begin{proof}
Condition \(D_m\) implies \(\depth \underline{\Omega}_U^p + p \geq n\) for all \(p\leq m\) by \cite[Lemma~4.2]{burke}, so the corollary follows immediately from Theorem \ref{thm lefschetz}.
\end{proof}

We now turn to the proof of Theorem~\ref{thm lefschetz}.

\begin{proof}[Proof of Theorem \ref{thm lefschetz}]
Let $j: U \hookrightarrow X$ denote the inclusion. There is a distinguished triangle
\begin{align*}
    j_! \mathbb{Q}_U^H \rightarrow \mathbb{Q}_X^H \rightarrow \mathbb{Q}_D^H \xrightarrow[]{+1}
\end{align*}
in the derived category $D^b(\mathrm{MHM}(X))$ of mixed Hodge modules. 

Taking hypercohomology gives a long exact sequence
\begin{align*}
    \ldots \rightarrow H^i_c(U, \mathbb{Q}) \rightarrow H^i(X, \mathbb{Q}) \rightarrow H^i(D, \mathbb{Q}) \rightarrow H^{i+1}_c(U, \mathbb{Q}) \rightarrow \ldots 
\end{align*}
It therefore suffices to prove $\gr_{-p}^F H_c^i(U, \mathbb{Q}) = 0$ whenever $i \leq k-1$ and $\min \{ p, q-1\} \leq m$. We establish this vanishing in the range $p \leq m$. The vanishing for $q \leq m+1$ then follows by duality, exactly as in the proof of \cite[Theorem $9.2$]{parkpopa}.

Saito \cite[Theorem $0.2$]{saito 00} identifies the Du Bois complex with the graded de Rham pieces of the trivial Hodge module, up to shift:
\[
    \underline{\Omega}_U^p \simeq \gr_{-p}^F \dr \big( \mathbb{Q}_U^H[n] \big)[p-n].
\]
Recall also that the depth of a complex is defined by $\depth \underline{\Omega}_U^p := - \max \{ j : \mathcal{H}^j \mathbb{D} \underline{\Omega}_U^p \neq 0 \}$, where $\mathbb{D}(-) = \mathbf{R} \mathcal{H}om(-, \omega_U^\bullet)$. Using standard properties of duality for Hodge modules, we compute
\begin{align*}
   \mathbb{D}  \underline{\Omega}_U^p \simeq& \mathbb{D} \bigg(\gr_{-p}^F \dr \big( \mathbb{Q}_U^H[n] \big)[p-n] \bigg) \\
   \simeq& \gr_p^F \dr \big( \mathbb{D} \mathbb{Q}_U^H \big) [-p].
\end{align*}
Therefore, the assumption $\depth \underline{\Omega}_U^p + p \geq k$ for all $p \leq m$ is equivalent to vanishing
\begin{align*}
    \mathcal{H}^j \gr_p^F \dr (\mathbb{D} \mathbb{Q}_U^H) = 0 \hspace{.5cm} \text{ for } j \geq -k+1 \text{ and } p \leq m.
\end{align*}
For a more detailed discussion of the depth of Du Bois complexes, see \cite[\S $4$]{burke}.

By \cite[Corollary~$6.3$]{burke}, the preceding vanishing is equivalent to
\begin{align*}
    \gr_p^F \dr (\mathcal{H}^j \mathbb{D} \mathbb{Q}_U^H) = 0 \hspace{.5cm} \text{ for } j \geq -k+1 \text{ and } p \leq m.
\end{align*}
By \cite[Lemma $3.4$]{park}, this vanishing is preserved under pushforward by the structure morphism $a_U: U \rightarrow \operatorname{Spec}\mathbb{C}$. Hence
\begin{align}\label{eqn depth vanishing}
    \gr_p^F \dr ( a_{U*} \mathcal{H}^j \mathbb{D} \mathbb{Q}_U^H) = 0 \hspace{.5cm} \text{ for } j \geq -k+1 \text{ and } p \leq m.
\end{align}
Finally, on a point the graded de Rham functor $\gr_p^F \dr(-)$ agrees with the usual graded functor $\gr_p^F(-)$. This is the key vanishing which we exploit below.

Recall that cohomology with compact support is computed by $H_c^i(U, \mathbb{Q}) = \mathcal{H}^i a_{U!} \mathbb{Q}_U^H$. By duality, this gives $\mathbb{D} H_c^i(U, \mathbb{Q}) \simeq \mathcal{H}^{-i} a_{U*} \mathbb{D} \mathbb{Q}_U^H$. Consider the spectral sequence
\begin{align*}
    E_2^{a,b} = \mathcal{H}^a a_{U*} \mathcal{H}^b \mathbb{D} \mathbb{Q}_U^H \Rightarrow \mathcal{H}^{a+b} a_{U*} \mathbb{D} \mathbb{Q}_U^H \simeq \mathbb{D} H_c^{-a-b}(U, \mathbb{Q}).
\end{align*}

Since $U$ is affine, $a_{U*}$ is a right exact functor by the generalized version of the Artin vanishing theorem \cite[Theorem $4.1.1$]{bbd}. Hence $E_2^{a,b} = 0$ for $a > 0$. On the other hand, \eqref{eqn depth vanishing} gives $\gr_p^F E_2^{a,b} = \mathcal{H}^a \gr_p^F a_{U*} \mathcal{H}^b \mathbb{D} \mathbb{Q}_U^H = 0$ whenever $b \geq -k+1$ and $p \leq m$.

Together, these imply $\gr_p^F E_2^{a,b} = 0$ for all $a+b \geq -k+1$ and $p \leq m$. Therefore, the spectral sequence shows that $\gr_{-p}^F H_c^i(U, \mathbb{Q}) = 0$ for $i \leq k-1$ and $p \leq m$, as claimed.
\end{proof}

Applying Theorem~\ref{thm lefschetz} successively to general hyperplane
sections yields the following bound on the weights of the cohomology of
$X$.

\begin{corollary}\label{cor depth lefschetz}
Let $X$ be an equidimensional projective variety and set $s = \dim \mathrm{Sing}(X)$. If $\depth \underline{\Omega}_X^p + p \geq k$ for all $p \leq m$, then 
\[
    H^i(X, \mathbb{Q}) \text{ has weights } \geq \min \{ 2m+2, i\} \text{ for all } i < k-s.
\]
\end{corollary}

\begin{proof}
Note first that if $U \subset X$ is an open subset, then $\depth \underline{\Omega}_U^p \geq \depth \underline{\Omega}_X^p$. Moreover, if $D \subset X$ is a general hyperplane section, then $\depth \underline{\Omega}_X^p + p \geq k$ for $p \leq m$ implies that 
\[
    \depth \underline{\Omega}_D^p+p\geq k-1
    \qquad\text{for all } p \leq m
\]
by \cite[Lemma $8.4$]{burke}. 

Let $Z$ be the intersection of $s+1$ general hyperplane sections of $X$. Applying Theorem~\ref{thm lefschetz} repeatedly, we obtain an injection
\begin{align*}
    H^{p,q}(X) \hookrightarrow H^{p,q}(Z) \hspace{.5cm} \text{ for } p+q<k-s \text{ and } p \leq m.
\end{align*}
Since $Z$ is smooth, its cohomology is pure. By strictness of the Hodge filtration, it follows that
\begin{align*}
    \gr_{-p}^F W_{i-1} H^i(X, \mathbb{Q}) = 0 \hspace{.5cm} \text{ for } i<k-s \text{ and } p \leq m.
\end{align*}
Now suppose that $2m+2 \leq i < k-s$. Since $W_{2m+1}H^i(X,\mathbb{Q})\subseteq W_{i-1}H^i(X,\mathbb{Q})$, we obtain
\begin{align*}
    \gr_{-p}^F W_{2m+1}H^i(X,\mathbb{Q})=0
    \qquad\text{for all } p\leq m.
\end{align*}
By duality, the vanishing of these Hodge-graded pieces forces $W_{2m+1}H^i(X,\mathbb{Q})=0$.
\end{proof}

We isolate the following consequence of Corollary \ref{cor depth lefschetz} for reference.

\begin{corollary}\label{cor O depth}
Let $X$ be an equidimensional projective variety and set $s = \dim \mathrm{Sing}(X)$. If $\depth \mathcal{O}_X \geq k$, then 
\[
    H^i(X, \mathbb{Q}) \text{ has weights } \geq \min \{ 2, i\} \text{ for all } i < k-s.
\]
\end{corollary}

\begin{proof}
This follows from Corollary~\ref{cor depth lefschetz}, together with the
inequality \( \depth \underline{\Omega}_X^0 \geq \depth \mathcal O_X \) of \cite[Theorem~3.3]{kovacs schwede}.
\end{proof}

For comparison, there is a stronger weight bound available
for varieties with rational singularities. This bound follows from Hodge
symmetry \cite[Theorem~A]{parkpopa} and was observed in
\cite[Corollary~1.7]{bps}.

\begin{lemma}\cite[Corollary~1.7]{bps}\label{lem weight restrictions}
Let $X$ be a projective variety with rational singularities. Then, for every
$i\geq 0$, the cohomology group \(H^i(X,\mathbb Q)\) has weights
\(\geq \min\{2,i\}\).
\end{lemma}

We need Corollary \ref{cor O depth} and Lemma \ref{lem weight restrictions} in the body of the paper to prove Proposition \ref{intro-prop weight zero}.

\end{document}